\documentclass{aptpub}
\usepackage{mathrsfs,upgreek}
\usepackage{cite}

\numberwithin{equation}{section}

\authornames{Dmitri Finkelshtein, Pasha Tkachov}
\shorttitle{Kesten's bound on the real line and its multi-dimensional analogues}

\newcommand{\eps}{\varepsilon}
\newcommand{\la}{\lambda}
\newcommand{\ga}{\gamma}
\newcommand{\ka}{\varkappa}
\newcommand{\R}{\mathbb{R}}
\newcommand{\X}{{\mathbb{R}^d}}
\renewcommand{\N}{\mathbb{N}}
\renewcommand{\B}{\mathcal{B}}
\newcommand{\1}{1\!\!1}
\DeclareMathOperator*{\essinf}{ess\,inf}
\newcommand\Dt[1][d]{{\mathcal{D}_{#1}}}
\newcommand\D[1][d]{{\widetilde{\mathcal{D}}_{#1}}}
\renewcommand\S[1][]{\mathcal{S}_{\mathrm{reg}\ifthenelse{\equal{#1}{}}{}{,}#1}}
\newcommand\St[1][d]{\widetilde{\mathcal{S}}_{\mathrm{reg},#1}}
\begin{document}

\title{Kesten's bound for sub-exponential densities on the real line and its multi-dimensional analogues}

\authorone[Swansea University]{Dmitri Finkelshtein}
\addressone{Department of Mathematics, Swansea University, Singleton Park, Swansea SA2~8PP, U.K. ({\tt d.l.finkelshtein@swansea.ac.uk})} 
\authortwo[Gran Sasso Science Institute]{Pasha Tkachov}
\addresstwo{Gran Sasso Science Institute, Viale Francesco Crispi, 7, 67100 L'Aquila AQ, Italy ({\tt pasha.tkachov@gssi.it})}

\begin{abstract}
We study the tail asymptotic of sub-exponential probability densities on the real line. Namely, we show that the $n$-fold convolution of a sub-exponential probability density on the real line is asymptotically equivalent to this density times $n$. We prove Kesten's bound, which gives a uniform in $n$ estimate of the $n$-fold convolution by the tail of the density. We also introduce a class of regular sub-exponential functions and use it to find an analogue of Kesten's bound for functions on $\X$. The results are applied for the study of the fundamental solution to a nonlocal heat-equation. 
\end{abstract}



\section{Introduction}
Let $F$ be a probability distribution on $\R$. Denote by $\overline{F}(s):=F\bigl((s,\infty)\bigr)$, $s\in\R$ its tail function. For probability distributions $F_1$, $F_2$ on $\R$, their convolution ${F}_1*{F}_2$ has the tail function 
\[
\overline{F_1*F_2}(s)=\int_\R \overline{F}_1(s-\tau)F_2(d\tau)=\int_\R \overline{F}_2(s-\tau)F_1(d\tau),
\]
where $\overline{F}_1, \overline{F}_2$ are the corresponding tail functions of $F_1,F_2$.

If a probability distribution $F$ is concentrated on $\R_+:=[0,\infty)$ and $\overline{F}(s)>0$, $s\in\R$, then, see e.g.~\cite{Chi1964},
\begin{equation}\label{liminfleq2}
\liminf_{s\to\infty} \frac{\overline{F*F}(s)}{\overline{F}(s)}\geq2.
\end{equation}
If, additionally, $F$ is heavy-tailed, i.e.~$\int_\R e^{\la s} F(ds)=\infty$ for all $\la>0$, 
then the equality holds in \eqref{liminfleq2}, see \cite{FoKo2007}. An important sub-class of heavy-tailed distributions concentrated on $\R_+$ constitute \emph{sub-exponential} ones, for which 
\begin{equation}\label{limeq2}
\lim_{s\to\infty} \frac{\overline{F*F}(s)}{\overline{F}(s)}=2.
\end{equation}
Any sub-exponential distribution on $\R_+$ is (right-side) long-tailed on $\R$, see e.g.~\cite{Chi1964}, i.e.~(cf. Definition~\ref{def:longtailedfunc} below)  
\begin{equation}\label{eq:longtailetail}
\lim_{s\to\infty} \frac{\overline{F}(s+t)}{\overline{F}(s)}=1 \quad \text{for each } t>0.
\end{equation}

If distributions $F_1,F_2$ on $\R$ have probability densities $f_1\geq0,f_2\geq0$, with $\int_\R f_1(s)\,ds=\int_\R f_2(s)\,ds=1$, then $F_1*F_2$ has the density 
\[
(f_1*f_2)(s):=\int_\R f_1(s-t)f_2(t)\,dt, \quad s\in\R.
\] 
The density $f$ of a sub-exponential distribution $F$ concentrated on $\R_+$ (i.e.~$f(s)=0$ for $s<0$) is said to be sub-exponential on $\R_+$ if $f$ is long-tailed, i.e.~\eqref{eq:longtailetail} holds with $\overline{F}$ replaced by $f$ (see also Definition~\ref{def:longtailedfunc} below), and, cf.~\eqref{limeq2},
\begin{equation}\label{subexpofdens2}
\lim_{s\to\infty} \frac{(f*f)(s)}{f(s)}=2.
\end{equation}
It can be shown (see e.g. \cite{AFK2003,FKZ2013,Klu1989}) that, in this case, for any $n\in\N$,
\begin{equation}\label{subexpofdensn}
\lim_{s\to\infty} \frac{f^{*n}(s)}{f(s)}=n,
\end{equation}
where $f^{*n}:=f*\ldots*f$ ($n-1$ times). Note that, in general, the density of a sub-exponential distribution concentrated on $\R_+$ even being long-tailed does not need to be a sub-exponential one; the corresponding characterisation can be found in \cite{AFK2003,FKZ2013}, see \eqref{eq:safsfa} below. 
The property \eqref{subexpofdensn} implies, in particular, that, for each $\delta>0$, $n\in\N$, there exists $s_n>0$, such that $f^{*n}(s)\leq (n+\delta)f(s)$ for $s>s_n$. In many situations, it is important to have similar inequalities `uniformly' in $n$, i.e.~on a set independent of $n$. A~possible solution is given by the so-called Kesten's bound, see \cite{Klu1989,AFK2003}: for a bounded sub-exponential density $f$ on $\R_+$ and for any $\delta>0$, there exist $c_\delta,s_\delta>0$, such that
\begin{equation}\label{eq:Kestenbdd}
 f^{*n}(s)\leq c_\delta (1+\delta)^n f(s), \quad s>s_\delta, \ n\in\N.
\end{equation}
For the corresponding results for distributions, see \cite{Chi1964,CNW1973,AN1972,FKZ2013}. Kesten's bounds were used to study series of convolutions of distributions on $\R_+$, $\sum_{n=1}^\infty \la_n F^{*n}$, and of the corresponding densities, $\sum_{n=1}^\infty \la_n f^{*n}$, 
appeared in different contexts: starting from the renewal theory (that was the motivation for the original paper \cite{Chi1964}) to branching age dependent processes, random walks, queue theory, risk theory and ruin probabilities, compound Poisson processes, and the study of infinitely divisible laws, see e.g.~\cite{EKM1997,AFK2003,FKZ2013,Wat2008,Klu1989,BB2008} and the references therein.  

If $F$ is a probability distribution on the whole $\R$, such that $F^+$, given by $F^+(B):=F(B\cap\R_+)$ for all Borel $B\subset\R$, is sub-exponential on $\R_+$, then (see e.g.~\cite[Lemma~3.4]{FKZ2013}) $F$ is long-tailed on $\R$ and \eqref{limeq2} holds. The distributions on $\R$ and their densities were considered by several authors, see \cite{Sgi1990,Sgi1982,RS1999,Wat2008} and others. The reference \cite{Wat2008}, in particular, gives a review of difficulties appeared in the case of the whole $\R$ and closes several gaps in the preceding results. However, even some basic properties of sub-exponential densities on the whole $\R$ remained open. 

\newpage

Namely, in \cite[Lemma 4.13]{FKZ2013}, it was shown that if an integrable on $\R$ function $f$\nopagebreak
\begin{itemize}
  \item is right-side long-tail and, being restricted on $\R_+$ and normalized in $L^1(\R_+)$, satisfies \eqref{subexpofdens2} (we will say then that $f$ is \emph{weekly sub-exponential} on $\R$, cf.~Definition~\ref{def:weaksubexponR} below), and if
  \item the condition 
  \begin{equation}\label{eq:condcond}
  f(s+\tau)\leq K f(s), \quad s>\rho, \ \tau>0
  \end{equation}
  holds, for some $K,\rho>0$ (in particular, if $f$ decays to $0$ at $\infty$, cf.~Definition~\ref{def:taildecreasing}), 
\end{itemize}
then \eqref{subexpofdens2} holds for the original $f$ on $\R$ as well. We generalize this to an analogue of \eqref{subexpofdensn} with a general $n\in\N$. In particular, we prove in Theorem~\ref{thm:maincor} below that
\begin{thm}\label{thm:intro1}
Let $f$  be an integrable weakly sub-exponential on $\R$ function, such that \eqref{eq:condcond} holds. Then $f$, being normalized in $L^1(\R)$, satisfies \eqref{subexpofdensn} for all $n\in\N$.
\end{thm}

Moreover, in Theorem~\ref{thm:KestenonR}, we prove that then \eqref{eq:Kestenbdd} holds as well. Namely, one has the following result.
\begin{thm}\label{thm:intro2}
Let $f$ be a bounded weakly sub-exponential probability density on $\R$, such that \eqref{eq:condcond} holds.
  Then, for each $\delta>0$, there exist $c_\delta, s_\delta>0$, such that \eqref{eq:Kestenbdd} holds.
\end{thm}

Note that the all `classical' examples of sub-exponential functions satisfy assumptions of Theorems~\ref{thm:intro1} and~\ref{thm:intro2}, see Subsection~\ref{subsec:examples}. 

The multi-dimensional version of the constructions above is much more non-trivial. Currently, there exist at least three different definitions of sub-exponential distributions on $\X$ for $d>1$, see \cite{CR1992,Ome2006,SS2016}. The variety is mainly related to  different possibilities to describe the zones in $\X$ where an analogue of the equivalence \eqref{limeq2} takes place. However, any results about sub-exponential densities in $\X$, $d>1$, seem to be absent at all.
Note also that if, e.g.~$a$ is radially symmetric, i.e.~$a(x)=b(|x|)$, $x\in\X$ (here $|x|$ denotes the Euclidean norm on $\X$) and $b$, being normalized, is a sub-exponential density on $\R_+$, then 
$(a*a)(x):=\int_\X a(x-y)a(y)\,dy=p(|x|)$, $x\in\X$, 
for some $p:\R_+\to\R_+$ (i.e.~$a*a$ is also radially symmetric), however, asymptotic behaviors of $b$ and $p$ at $\infty$ are hardly to be compared. Leaving this problem as on open, we focus in this paper on an analogue of Kesten's bound \eqref{eq:Kestenbdd} in the multi-dimensional case. 

To do this, we introduce a special class $\St[d]$ of regular sub-exponential functions on $\R_+$ (see~Definitions~\ref{def:super-subexp} and~\ref{def:super-subexp-tilde}). Functions from this class are either inverse polynomials (i.e.~\eqref{ass:newA5} holds), or decay at $\infty$ faster than any polynomial (i.e.~\eqref{eq:quicklydecreasing} holds), but slower than any exponential function, with the fastest allowed asymptotic $\exp\bigl(-s(\log s)^{-q}\bigr)$ with $q>1$, cf.~Remark~\ref{rem:exception}. Then, in Corollary~\ref{cor:mainRd}, we show the following result.
\begin{thm}\label{thm:intro3}
Let $a=a(x)$ be a probability density on $\X$, such that $a(x)=b(|x|)$, $x\in\X$, for some $b\in\St[d]$. Then, for each $\delta>0$ and for each $\alpha<1$ close enough~to~$1$, 
\begin{equation}\label{eq:Kestenbddmulti}
 a^{*n}(x)\leq c_{\alpha} (1+\delta)^n a(x)^\alpha, \quad |x|>s_\alpha, \ n\in\N
\end{equation}
for some $c_{\alpha}=c_\alpha(\delta)>0$ and $s_\alpha=s_\alpha(\delta)>0$.
\end{thm}

Clearly, $a(x)=o\bigl(a(x)^\alpha\bigr)$, $|x|\to\infty$, for any $\alpha\in(0,1)$, hence the inequality \eqref{eq:Kestenbddmulti} is weaker than \eqref{eq:Kestenbdd} for the case $d=1$.

The results of Corollary~\ref{cor:mainRd} is based on more general Theorem~\ref{thm:mainRd}, which says that if, for some $b\in\St[d]$ and decreasing on $\R_+$ function $p$,
  \begin{equation*}
  a(x)\leq p(|x|), \quad x\in\X, \ \qquad\ \log p (s)\sim \log b(s), \quad s\to\infty,
  \end{equation*}  
then \eqref{eq:Kestenbddmulti} holds with $a(x)$ replaced by $b(|x|)$ in the right hand side.
 
The paper is organized as follows. In Section~\ref{sec:techtoolsonR}, we consider properties of general sub-exponential functions on the real line and prove the results which imply Theorems~\ref{thm:intro1} and~\ref{thm:intro2}. In Section~\ref{sec:regsubexp}, we define and study properties of regular sub-exponential functions on the real line and consider the corresponding examples. In Section~\ref{sec:KbddRd}, we prove Theorem~\ref{thm:intro3} and its generalizations. Finally, in Appendix, we apply the obtained results to the study of a non-local heat equation.

\section{Sub-exponential functions and Kesten's bound on the real line}\label{sec:techtoolsonR}
\begin{defn}\label{def:longtailedfunc}
        A function $b:\R\to\R_+$ is said to be {\em (right-side) long-tailed} if there exists $\rho\geq0$, such that $b(s)>0$, $s\geq\rho$; and, for any $\tau\geq 0$,
                \begin{equation}\label{eq:longtaileddef}
                        \lim_{s\to\infty}\frac{b(s+\tau)}{b(s)}=1.
                \end{equation}
\end{defn}

\begin{rem}
        By \cite[formula (2.18)]{FKZ2013}, the convergence in \eqref{eq:longtaileddef} is equivalent to the locally uniform in $\tau$ convergence, namely, \eqref{eq:longtaileddef} can be replaced by the assumption that, for all $h>0$,
        \begin{equation}\label{eq:longtaileddef-uniform}
                        \lim_{s\to\infty}\sup_{|\tau|\leq h}\biggl\lvert\frac{b(s+\tau)}{b(s)}-1\biggr\rvert= 0.
                \end{equation}
\end{rem}

A long-tailed function has to have a `heavier' tail than any exponential function; namely, the following statement holds.

\begin{lem}[\!\!\!{\cite[Lemma 2.17]{FKZ2013}}]\label{eq:slowlythanexp}
Let $b:\R\to\R_+$ be a long-tailed function. Then, for any $k>0$,   $\lim\limits_{s\to\infty} e^{ks}b(s)=\infty$. 
\end{lem}

The constant $h$ in \eqref{eq:longtaileddef-uniform} may be arbitrary big. It is quite natural to ask what will be if $h$ increases to $\infty$ consistently with $s$.

\begin{lem}[{Cf.~\cite[Lemma 2.19, Proposition 2.20]{FKZ2013}}]
Let $b:\R\to\R_+$ be a long-tailed function. Then there exists a function $h:(0,\infty)\to(0,\infty)$, with
$h(s)<\frac{s}{2}$ and $\lim\limits_{s\to\infty}h(s)=\infty$, such that, cf.~\eqref{eq:longtaileddef-uniform},
\begin{equation}\label{eq:uniformlongtailedR}
  \lim_{s\to\infty}\sup_{|\tau|\leq h(s)}\biggl\lvert\frac{b(s+\tau)}{b(s)}-1\biggr\rvert= 0.
\end{equation}
\end{lem}

Following \cite{FKZ2013}, we will say then that $b$ is $h$-insensitive. Of course, for a given long-tailed function $b$, the function $h$ that fulfills \eqref{eq:uniformlongtailedR} is not unique, see also \cite[Proposition~2.20]{FKZ2013}.

The convergence in \eqref{eq:longtaileddef} will not be, in general, monotone in $s$. To get this monotonicity, we consider the following class of functions.

\begin{defn}
    A function $b:\R\to\R_+$ is said to be {\em (right-side) tail-log-convex}, if there exists $\rho>0$, such that $b(s)>0$, $s\geq \rho$, and the function $\log b$ is convex on $[\rho,\infty)$.
\end{defn}
\begin{rem}
  It is well-known that any function which is convex on an open interval is continuous there. Therefore, a tail-log-convex function $b=\exp(\log b)$ is continuous on $(\rho_b,\infty)$ as well.
\end{rem}

\begin{lem}
        Let $b:\R\to\R_+$ be tail-log-convex, with $\rho=\rho_b$. Then, for any $\tau>0$, the function $\frac{b(s+\tau)}{b(s)}$ is non-decreasing in $s\in[\rho,\infty)$.
\end{lem}
\begin{proof}
Take any $s_1>s_2\geq\rho$. Set $B(s):=\log b(s)\leq 0$, $s\in[\rho,\infty)$. Then the desired inequality $  \frac{b(s_1+\tau)}{b(s_1)}\geq\frac{b(s_2+\tau)}{b(s_2)}$
  is equivalent to
  $  B(s_1+\tau)+B(s_2)\geq B(s_2+\tau)+B(s_1)$.
  Since $B$ is convex, we have, for $\la=\frac{\tau}{s_1-s_2+\tau}\in(0,1)$,
  \begin{align*}
  B(s_1)=B\bigl(\la s_2+(1-\la)(s_1+\tau)\bigr)\leq \la B(s_2)+(1-\la)B(s_1+\tau), \\
  B(s_2+\tau)=B\bigl((1-\la)s_2+\la(s_1+\tau)\bigr)\leq (1-\la) B(s_2)+\la B(s_1+\tau),
  \end{align*}
  that implies the needed inequality.
\end{proof}

Because of the terminology mentioned in the introduction, we will use the following definition.
\begin{defn}\label{def:weaksubexponR}
 We will say that a function $b:\R\to\R_+$ is {\em weakly (right-side)  sub-exponential on~$\R$} if $b$ is long-tailed, $b\in L^1(\R_+)$, and the function 
\begin{equation}\label{eq:defofbplus}
      b_+(s):=\1_{\R_+}(s)\biggl(\int_{\R_+}b(\tau)d\tau\biggr)^{-1}b(s), \quad s\in\R,
  \end{equation}
  satisfies the following asymptotic relation (as $s\to\infty$)
  \begin{equation}
    (b_+*b_+)(s)=\int_\R b_+(s-\tau) b_+(\tau) \,d\tau=\int_0^s b_+(s-\tau) b_+(\tau) \,d\tau\sim 2 b_+(s).
    \label{eq:defofsubexponR+}
  \end{equation}
\end{defn}

The next statement shows that a long-tailed tail-log-convex function is weakly sub-exponential on $\R$ provided that it decays at $\infty$ fast enough.
\begin{lem}[{cf.~\cite[Theorem~4.15]{FKZ2013}}]\label{le:SimplsubexpR+}
  Let $b:\R\to\R_+$ be a long-tailed tail-log-convex function such that $b\in L^1(\R_+)$. Suppose that, for a function $h:(0,\infty)\to(0,\infty)$, with
$h(s)<\frac{s}{2}$ and $\lim\limits_{s\to\infty}h(s)=\infty$, the asymptotic \eqref{eq:uniformlongtailedR} holds, and 
\begin{equation}\label{eq:S0}
  \lim_{s\to\infty}s\, b\bigl(h(s)\bigr)=0.
\end{equation}
Then $b$ is weakly sub-exponential on~$\R$.
\end{lem}

\begin{rem}
Let $b:\R\to\R_+$ be a weakly sub-exponential function on~$\R$. Then, by \eqref{eq:defofbplus}, \eqref{eq:defofsubexponR+}, we have
    \begin{equation}\label{eq:tocompare}
    \int_0^s b(s-\tau) b(\tau) \,d\tau\sim 2 \biggl(\int_{\R_+}b(\tau)d\tau\biggr) b(s), \quad s\to\infty.
    \end{equation}
\end{rem}

\begin{defn}
 We will say that a function $b:\R\to\R_+$ is {\em (right-side)  sub-exponential on~$\R$} if $b$ is long-tailed, $b\in L^1(\R)$, and the following asymptotic relation holds, cf.~\eqref{eq:defofsubexponR+}, \eqref{eq:tocompare},
\begin{equation}\label{eq:realsubexponR}
  (b*b)(s)=\int_\R b(s-\tau) b(\tau)\,d\tau\sim 2 \biggl(\int_{\R}b(\tau)d\tau\biggr) b(s), \quad s\to\infty.
\end{equation}
\end{defn}

\begin{rem}
  By \cite[Lemma~4.12]{FKZ2013}, a sub-exponential function on $\R$ is weakly sub-exponential there. The following lemma presents a sufficient condition to get the converse.
\end{rem}

\begin{lem}[cf.~{\cite[Lemma~4.13]{FKZ2013}}]\label{lem:repeatFoss}
  Let $b\in L^1(\R\to\R_+)$ be a weakly sub-exponential function on $\R$. Suppose that there exists $\rho=\rho_b>0$ and $K=K_b>0$ such that
\begin{equation}\label{eq:quasi-decreasing}
  b(s+\tau)\leq K b(s), \quad s>\rho, \ \tau>0.
\end{equation}
Then \eqref{eq:realsubexponR} holds, i.e.~$b$ is 
sub-exponential on $\R$.
\end{lem}
\begin{rem}\label{rem:quasi-decreasing_tends_to_0}
	For $b\in L^1(\R\to\R_+)$, condition \eqref{eq:quasi-decreasing} yields that $\sup_{t\geq s} b(t) \to 0$, $s\to\infty$. In particular $b(s)\to 0$, as $s\to\infty$. 
\end{rem}
An evident sufficient condition which ensures \eqref{eq:quasi-decreasing} is that $b$ is decreasing on $[\rho,\infty)$. Consider the corresponding definition.
\begin{defn}\label{def:taildecreasing}
  A function $b:\R\to\R_+$ is said to be {\em (right-side) tail-decreasing} if there exists a number $\rho=\rho_b\geq0$ such that $b=b(s)$ is strictly decreasing on $[\rho,\infty)$ to~$0$. In particular, $b(s)>0$, $s\geq\rho$.
\end{defn}

The proof of the following useful statement is straightforward.
\begin{prop}\label{prop:useful}
Let $b:\R\to\R_+$ be a tail-decreasing function. Let $h:(0,\infty)\to(0,\infty)$, with $h(s)<\frac{s}{2}$ and $\lim\limits_{s\to\infty}h(s)=\infty$. Then \eqref{eq:uniformlongtailedR} holds, if and only if
\begin{equation}\label{eq:surprise}
\lim_{s\to\infty}\frac{b(s\pm h(s))}{b(s)}=1.
\end{equation}
\end{prop}

The next statement and its proof follow ideas of \cite[Lemma 4.13, Lemma 4.9]{FKZ2013}.
\begin{prop}\label{prop:almostdream}
Let $b:\R\to\R_+$ be a weakly sub-exponential function on $\R$, such that \eqref{eq:quasi-decreasing} holds. Let $b_1,b_2\in L^1(\R\to\R_+)$
and there exist $c_1,c_2\geq0$, such that
\begin{equation}\label{eq:cond2}
\lim\limits_{s\to\infty}\frac{b_j(s)}{b(s)}= c_j , \quad j=1,2.
\end{equation}
Then
\begin{equation}\label{eq:dream}
  \lim\limits_{s\to\infty}\frac{(b_1*b_2)(s)}{b(s)}= c_1\int_\R b_2(\tau)\,d\tau +c_2\int_\R b_1(\tau)\,d\tau.
\end{equation}
\end{prop}
\begin{proof}
  Since $b_+$, given by \eqref{eq:defofbplus}, is long-tailed, and \eqref{eq:defofsubexponR+} holds, we have, by \cite[Theorem 4.7]{FKZ2013}, that there exists an increasing function $h:(0,\infty)\to(0,\infty)$, such that $h(s)<\frac{s}{2}$, $\lim\limits_{s\to\infty}h(s)=\infty$, and
\begin{equation}\label{eq:safsfa}
  \int_{h(s)}^{s-h(s)}b_+(s-\tau) b_+(\tau)\,d\tau=o(b_+(s)), \quad s\to\infty,
\end{equation}
and, evidently, one can replace $b_+$ by $b$ in \eqref{eq:safsfa}.

 Next, for any $b_1,b_2\in L^1(\R\to\R_+)$, one can easily get that
  \begin{align*}
    (b_1*b_2)(s) &=\int_{-\infty}^{h(s)} \bigl(b_1(s-\tau)b_2(\tau)+b_1(\tau)b_2(s-\tau)\bigr)\,d\tau +\int_{h(s)}^{s-h(s)} b_1(s-\tau)b_2(\tau)\,d\tau.
  \end{align*}
Take an arbitrary $\delta\in(0,1)$. By~\eqref{eq:cond2}, \eqref{eq:uniformlongtailedR}, and \eqref{eq:safsfa} (the latter, with $b_+$ replaced by $b$), there exist $K,\rho>0$, such that \eqref{eq:quasi-decreasing} holds and, for all $s\geq h(\rho)$,
  \begin{gather}\label{eq:newdop}
    \bigl\lvert b_j(s)-c_jb(s)\bigr\rvert+\sup_{|\tau|\leq h(s)}\bigl\lvert b(s+\tau)-b(s)\bigr\rvert+\int_{h(s)}^{s-h(s)}b(s-\tau) b(\tau)\,d\tau\leq\delta b(s),\\
    \int_{-\infty}^{-h(s)} b_j(\tau)d\tau+
    \int_{h(s)}^{\infty} b_j(\tau)d\tau\leq \delta, \quad j=1,2.
    \label{eq:newdop4}
  \end{gather}
Then, by \eqref{eq:quasi-decreasing}, \eqref{eq:newdop}, \eqref{eq:newdop4}, for $s\geq\rho>h(\rho)$,
  \begin{align*}
    &\quad\int_{-\infty}^{-h(s)} \bigl(b_1(s-\tau)b_2(\tau)+b_1(\tau)b_2(s-\tau)\bigr)\,d\tau \leq \delta K (c_1+c_2+2\delta) b(s).
  \end{align*}
Set $B_j:=\int_\R b_j(s)\,ds$, $j=1,2$. By \eqref{eq:newdop}, for $s\geq \rho$, 
  \begin{gather*}
    \quad\biggl\lvert \int_{-h(s)}^{h(s)} b_1(s-\tau)b_2(\tau)\,d\tau
    -c_1 b(s)\int_{-h(s)}^{h(s)} b_2(\tau)\,d\tau\biggr\rvert \leq \delta B_2(1+c_1)b(s),\\
    \int_{h(s)}^{s-h(s)} b_1(s-\tau)b_2(\tau)\,d\tau \leq \delta (c_1+\delta)(c_2+\delta) b(s).
  \end{gather*}

 From the obtained estimates, it is straightforward to get that, for some $M>0$, $\bigl\lvert (b_1*b_2)(s)-(c_1B_2+c_2B_1)b(s)\bigr\rvert\leq \delta M b(s)$ for $s\geq\rho$.
  The latter implies \eqref{eq:dream}.
\end{proof}
\begin{cor}\label{cor:right-side_is_right-side}
	The property of an integrable on $\R$ function to be weekly sub-exponential on $\R$ depends on its tail property  only. Namely, for a weakly sub-exponential on $\R$ function $b\in L^1(\R\to\R_+)$ and for any $c\in L^1(\R\to\R_+)$ and $s_0\in\R$, the function $\tilde{b}(s) = \1_{(-\infty,s_0)}(s) c(s) + \1_{[s_0,\infty)}(s) b(s)$ is weakly sub-exponential on $\R$, cf.~also Theorem~\ref{thm:coolthm} below.
\end{cor}

Now one gets a generalization of Lemma~\ref{lem:repeatFoss}.

\begin{thm}\label{thm:maincor}
Let $b\in L^1(\R\to\R_+)$ be a weakly sub-exponential on $\R$ function, such that \eqref{eq:quasi-decreasing} holds (for example, let $b$ be tail-decreasing).
Then
  \begin{equation}\label{eq:subexpproperty}
    \lim_{s\to\infty} \frac{b^{*n}(s)}{b(s)}=n\Bigl(\int_\R b(\tau)\,d\tau\Bigr)^{n-1}, \quad n\geq2,
  \end{equation}
  where $b^{*n}(s)=(\underbrace{b*\ldots * b}_n)(s)$, $s\in\R$.
\end{thm}
\begin{proof}
  Take in Proposition~\ref{prop:almostdream}, $b_1=b_2=b$, i.e.~$c_1=c_2=1$. Then, for $B:=\int_\R b(\tau)\,d\tau$, one gets $b^{*2}(s)\sim 2B b(s)$, $s\to\infty$.
  Proving by induction, assume that $b^{*(n-1)}(s)\sim (n-1)B^{n-2} b(s)$, $s\to\infty$, $n\geq 3$. Take in Proposition~\ref{prop:almostdream}, $b_1=b$, $b_2=b^{*(n-1)}$, $c_1=1$, $c_2=(n-1)B^{n-2} $, then, since $\int_\R b^{*(n-1)}(\tau)\,d\tau =B^{n-1}$, one gets
    $\lim\limits_{s\to\infty}\frac{b^{*n}(s)}{b(s)}=B^{n-1}+B(n-1)B^{(n-2)} =nB^{(n-1)}$.
\end{proof}

Consider now some general statements in the Euclidean space  $\X$, $d\in\N$. Fix the Borel $\sigma$-algebra $\B(\X)$ there. All functions on $\X$ in the sequel are supposed to be $\B(\X)$-measurable. Let $0\leq a\in L^1(\X)$ be a fixed probability density on $\X$, i.e.~
\begin{equation}
\int_\X a(x)\,dx=1. \label{eq:normed}
\end{equation}
 Let $f:\R\to\R$; we will say that the convolution
\[
(a*f)(x):=\int_\X a(x-y)f(y)\,dy, \quad x\in\X
\]
is well-defined if the function $y\mapsto a(x-y)f(y)$ belongs to $L^1(\X)$ for a.a. $x\in\X$. In particular, this holds if $f\in L^\infty(\X)$. 
Next, for a function $\phi:\X\to(0,+\infty)$, we define, for any $f:\X\to\R$,
\begin{equation*}
  \lVert f\rVert_\phi:=\sup_{x\in\X} \frac{|f(x)|}{\phi(x)}\in[0,\infty].
\end{equation*}

\begin{prop}\label{prop:quasiKersten}
 Let a function $\phi:\X\to(0,+\infty)$ be such that $a*\phi$ is well-defined,
 $\lVert a\rVert_\phi<\infty$, and, for some $\ga\in(0,\infty)$,
\begin{equation} \label{eq:ba:Phi_est_suff_cond}
\frac{(a*\phi)(x)}{\phi(x)}\leq\ga, \quad x\in\X.
\end{equation}
Then $a^{*n}(x)\leq \ga^{n-1} \lVert a\rVert_\phi \phi(x)$, $x\in\X$.
\end{prop}
\begin{proof}
For any $f:\X\to\R$, with $\lVert f\rVert_\phi<\infty$, we have, for $x\in\X$,
\begin{equation*}
  \biggl\lvert \frac{(a*f)(x)}{\phi(x)}\biggr\rvert\leq \int_{\X}\frac{a(y)\phi(x-y)}{\phi(x)}\frac{|f(x-y)|}{\phi(x-y)}dy\leq \frac{a*\phi(x)}{\phi(x)} \lVert f\rVert_\phi\leq \ga \lVert f\rVert_\phi.
\end{equation*}
In particular, since $\lVert a\rVert_\phi<\infty$, one gets 
$\lVert a*a\rVert_\phi\leq \ga \lVert a\rVert_\phi<\infty$.
Proceeding inductively, one gets $\lVert a^{*n}\rVert_\phi\leq \ga \lVert a*a^{n-1}\rVert_\phi\leq \ga^{n-1}\lVert a\rVert_\phi<\infty$,
that yields the statement.
\end{proof}

\begin{prop}\label{prop:ba:suff_cond1}
  Let a function $\upomega:\X\to(0,+\infty)$ be such that, for any $\la>0$, 
  \begin{equation}\label{eq:defofsetOmegala}
    \Omega_\la:=\Omega_\la(\upomega):=\bigl\{x\in\X: \upomega(x)<\la\bigr\}\neq\emptyset.
  \end{equation}
Suppose further 
\begin{equation} \label{eq:ba:23r23Phi_est_suff_cond}
\eta:=\limsup_{\la\to0+}\sup_{x\in\Omega_\la}\frac{( a *\upomega)(x)}{\upomega(x)}\in(0,\infty).
\end{equation}
Then, for any $\delta\in(0,1)$, there exists $\la=\la(\delta,\upomega)\in(0,1)$, such that \eqref{eq:ba:Phi_est_suff_cond} holds, with
 \begin{equation}\label{eq:defofomegala}
    \phi(x):=\upomega_\la(x):=\min\bigl\{\la,\upomega(x)\bigr\}, \quad x\in\X,
  \end{equation}
and $\ga:=\max\{1,(1+\delta)\eta\}$.
\end{prop}

\begin{proof}
By \eqref{eq:defofomegala}, for an arbitrary $\la>0$, we have $\upomega_\la(x)\leq \la$, $x\in\X$; then $(a*\upomega_\la)(x)\leq \la$, $x\in\X$, as well. In particular, cf.~\eqref{eq:defofomegala},
\begin{equation}\label{eq:ineq1213}
    (a*\upomega_\la)(x)\leq
    \upomega_\la(x), \quad x\in\X\setminus\Omega_\la.
  \end{equation}
Next, by \eqref{eq:ba:23r23Phi_est_suff_cond}, for any $\delta>0$ there exists $\la=\la(\delta)\in(0,1)$ such that
\[
\sup_{x\in\Omega_\la}\frac{( a *\upomega)(x)}{\upomega(x)}-\eta\leq \delta \eta,
\]
in particular, $(a*\upomega)(x)\leq (1+\delta)\eta \upomega(x)=(1+\delta)\eta \upomega_\la(x), \quad x\in\Omega_\la$.
  Therefore,
  \begin{equation}\label{eq:ineq1214}
        (a*\upomega_\la)(x)=(a*\upomega)(x)-\left( a*(\upomega-\upomega_\la) \right)(x)\leq (1+\delta)\eta \upomega_\la(x),
  \end{equation}
  for all $x\in\Omega_\la$; here we used that $\upomega\geq \upomega_\la$. Then \eqref{eq:ineq1213}--\eqref{eq:ineq1214} yield the statement.
\end{proof}

We are ready to prove now Kesten's bound on $\R$.

\begin{thm}\label{thm:KestenonR}
  Let $b\in L^1(\R\to\R_+)$ be a bounded weakly sub-exponential on $\R$ function with $\int_\R b(s)\,ds=1$, such that \eqref{eq:quasi-decreasing} holds. 
  Then, for any $\delta\in(0,1)$, there exist $C_\delta,s_\delta>0$, such that 
\begin{equation}\label{eq:Kb1}
  b^{*n}(s)\leq C_\delta (1+\delta)^{n} b(s), \quad s>s_\delta, n\in\N.
\end{equation}
\end{thm}
\begin{proof}
Fix $\delta\in(0,1)$ and  $\eps\in\bigl(0,\delta]$ with $(1+\eps)^3\leq 1+\frac{\delta}{2}$. Let $s_1,\la_1>0$ satisfy
\begin{equation}
\int_{-\infty}^{-s_1} b(s)ds+\int_{s_1}^{\infty} b(s)ds\leq \frac{\eps}{2}, \qquad 4s_1\la_1 \leq \eps.
   \label{eq:small_perturbation_around_infinity}
\end{equation}
Define the following functions, for $s\in\R$,
  \begin{equation}\label{eq:newfuncdef}
  \begin{gathered}
    \tilde{b}(s) := \1_{(-\infty,-s_1)}(s) b(-s)+ \1_{[-s_1,s_1]}(s) \max\{\la_1,b(s)\} + \1_{(s_1,\infty)}(s) b(s),\\
    b_1(s):=\1_{(-\infty,-s_1)}(s) b(s), \qquad a(s):=\frac{\tilde{b}(s)}{\|\tilde{b}\|_1}.
   \end{gathered}
  \end{equation}  
Here and below, $\|\cdot\|_1$ denotes the norm in $L^1(\R)$. Then, by \eqref{eq:small_perturbation_around_infinity}, \eqref{eq:newfuncdef},  $\|b-\tilde{b}\|_1\leq \eps$, and hence
  \begin{equation}
\| \tilde{b}\|_1 \leq 1+\eps.\label{eq:tildebnormest}
\end{equation}

By Corollary \ref{cor:right-side_is_right-side}, both functions $s\mapsto\tilde{b}(s)$ and $s\mapsto\tilde{b}(-s)$ are weakly sub-exponential on $\R$.
Hence by Theorem \ref{thm:maincor}, 
  $\lim\limits_{s\to\pm\infty} \frac{\tilde{b}^{*m}(s)}{\tilde{b}^{*k}(s)} 
    = \frac{m}{k} \|\tilde{b}\|_1^{m-k}$ for any $k,m\in\N$.
  In particular, there exist $m_0\in\N$ and $s_2\geq s_1$, such that, for $\upomega := \tilde{b}^{*m_0}$ and $|s|\geq s_2$,
  \begin{gather}
    \frac{\tilde{b}*\upomega(s)}{\upomega(s)} \leq (1+\eps) \|\tilde{b}\|_1,\label{eq:another_suff_cond} \\
    m_0 (1-\eps) \|\tilde{b}\|_1^{m_0-1} \leq \frac{\upomega(s)}{\tilde{b}(s)} \leq m_0 (1+\eps) \|\tilde{b}\|_1^{m_0-1}. \label{eq:omega_vs_tilde_b}
  \end{gather}

  For an $r>0$, $\essinf\limits_{s\in[-r,r]}\tilde{b}(s)>0$. It is straightforward to check then by induction, that
  \begin{equation}\label{eq:omega_is_separated_from_zero}
    \essinf\limits_{-r\leq s \leq r}\upomega(s)=\essinf\limits_{-r\leq s \leq r}\tilde{b}^{*m_0}(s)>0, \qquad r>0.
  \end{equation}
On the other hand, by Remark \ref{rem:quasi-decreasing_tends_to_0} and \eqref{eq:omega_vs_tilde_b}, we have
  \begin{equation}\label{eq:qrwrwqqwrt34435}
 \lim_{s\to\pm\infty} \upomega(s) = \lim_{s\to\pm\infty} \tilde{b}(s) = 0. 
  \end{equation} 

Hence there exists $\la_2\in (0, \la_1]$, such that for any $\la\in(0,\la_2)$, the set $\Omega_\la$ defined by \eqref{eq:defofsetOmegala} will be a non-empty subset of $(-\infty,-s_\la)\cup(s_\la,\infty)$, where $s_\la\to\infty$, as $\la\searrow 0$. Therefore, by \eqref{eq:another_suff_cond}, the condition \eqref{eq:ba:23r23Phi_est_suff_cond} holds with $a$ given by \eqref{eq:newfuncdef} and $\eta\leq 1+\eps$. Then, by Proposition \ref{prop:ba:suff_cond1}, there exists $\la_3\in(0,\min\{\la_2,1\})$, such that \eqref{eq:ba:Phi_est_suff_cond} holds with 
\[
\phi(s):=\min\{\la_3,\upomega(s)\}, \quad s\in\R, \qquad \gamma:=(1+\eps)^2.
\]
By \eqref{eq:omega_vs_tilde_b}, \eqref{eq:omega_is_separated_from_zero} and since $\tilde{b}$ is bounded, we have $\|\tilde{b}\|_\phi<\infty$. Hence, by Proposition~\ref{prop:quasiKersten} and using that $\gamma \|\tilde{b}\|_1\leq 1+\frac{\delta}{2}$ because of \eqref{eq:tildebnormest} and the choice of $\eps$, one easily gets
\begin{equation*}
    \tilde{b}^{*n}(s) \leq \Bigl(1+\frac{\delta}{2}\Bigr)^{n-1}  \|\tilde{b}\|_\phi \min\{ \la_3, \upomega(s)\}, \qquad s\in\R,\ n\in \N.
\end{equation*}
By \eqref{eq:qrwrwqqwrt34435}, \eqref{eq:omega_vs_tilde_b}, $\|\tilde{b}\|_\phi \min\{ \la_3, \upomega(s)\}\leq C_\delta\tilde{b}(s)=C_\delta b(s)$ for $s>s_3$, with some $s_3\geq s_2$ and $C_\delta=C_\delta(m_0)>0$. As a result,
\begin{equation}\label{eq:foruseful}
    \tilde{b}^{*n}(s) \leq C_\delta \Bigl(1+\frac{\delta}{2}\Bigr)^{n} b(s), \qquad s>s_3,\ n\in \N.
\end{equation}

By \eqref{eq:newfuncdef}, $b\leq b_1+\tilde{b}$ and hence $ b^{*n}\leq \sum_{k=0}^{n} \binom{n}{k}b_1^{*k} * \tilde{b}^{*(n-k)}$, $n\in\N$, pointwise. 
    Since $b_1^{*k}(s)=0$ for all $s\geq -s_1$, $k\in\N$, then, by \eqref{eq:foruseful}, one gets, for $s\geq s_3$,
    \begin{align*}
    &\quad  b_1^{*k} *\tilde{b}^{*(n-k)}(s) = \int_{s+s_1}^\infty b_1^{*k}(s-y) \tilde{b}^{*(n-k)}(y)dy\\
      &\leq C_\delta \Bigl(1+\frac{\delta}{2}\Bigr)^{n-k} \int_{s+s_1}^\infty b_1^{*k}(s-y) \tilde{b}(y)dy =C_\delta\Bigl(1+\frac{\delta}{2}\Bigr)^{n-k} \int_{-\infty}^{-s_1} b_1^{*k}(y)\tilde{b}(s-y)dy,
 \intertext{and by (2.9), (2.25), and (2.27), one can continue, for $s\geq s_\delta:=\max\{s_3,\rho\}$,}
&\leq       C_\delta \Bigl(1+\frac{\delta}{2}\Bigr)^{n-k} \Bigl(\frac{\eps}{2}\Bigr)^k K \tilde{b}(s)  \leq C_\delta  \Bigl(1+\frac{\delta}{2}\Bigr)^{n-k} \Bigl(\frac{\delta}{2}\Bigr)^k b(s),
  \end{align*}  
that yields \eqref{eq:Kb1}.
\end{proof}

\begin{rem}
Following the scheme of the proof for \cite[Proposition~8]{AFK2003}, it may be shown that Kesten's bound \eqref{eq:Kb1} holds under a weaker assummption that $b\in L^1(\R\to\R_+)$ is a bounded on $\R$ function with $\int_\R b(s)\,ds=1$ such that  \eqref{eq:subexpproperty} holds, i.e. 
\[
b^{*n}(x)\sim n b(x), \quad x\to\infty, \ n\geq 2
\]
(recall that, in contrast to \eqref{subexpofdensn} $b$ is not necessary concentrated on $\R_+$).
By Theorem~\ref{thm:maincor}, a sufficient condition for the latter asymptotic relation is that $b$ is weakly sub-exponential (i.e., cf. \eqref{eq:defofbplus}--\eqref{eq:defofsubexponR+}, its normalised restriction $b_+$ on $\R_+$ is subexponential) and the inequality \eqref{eq:quasi-decreasing} holds. 
\end{rem}

\section{Regular sub-exponential densities on $\R$}\label{sec:regsubexp}

We are going to obtain an analogue of Kesten's bound on $\X$ with $d>1$, at least for radially symmetric functions. Our technique will require to deal with functions $b(|x|)^\alpha$, $x\in\X$, where $b$ is a sub-exponential function on $\R_+$ and $\alpha<1$ is close enough to $1$; in particular, we have to be sure that $b^\alpha$ is also sub-exponential on $\R_+$. Moreover, to weaken the condition of radial symmetry, we will allow double-side estimates by functions of the form $p(|x|)b(|x|)$ for appropriate $p$ on $\R_+$ (say, polynomial). Again, we will need have to check whether the functions $pb$ is also sub-exponential on $\R_+$. To check such a stability of the class of sub-exponential on $\R_+$ functions with respect to power and multiplicative perturbations, we have to reduce the class to appropriately regular sub-exponential functions. Then the mentioned stability takes place, see Theorem~\ref{thm:coolthm} and Proposition~\ref{prop:tailprop2cool}. The examples of regular sub-exponential functions are given in Subsection~\ref{subsec:examples}. The analogues of Kesten's bound on $\X$ are considered in Section~\ref{sec:KbddRd}.

\subsection{Main properties}
\begin{defn}\label{def:super-subexp}
  Let $\S$ be the set of all functions $b:\R\to\R_+$ such that
  1)~$b\in L^1(\R_+)$ and $b$ is bounded on $\R$;
  2)~there exists $\rho=\rho_b>1$, such that $b$ is log-convex and strictly decreasing to~$0$ on $[\rho,\infty)$ (i.e.~$b$ is simultaneously tail-decreasing and tail-log-convex), and  $b(\rho)\leq 1$ (without loss of generality);
  3)~there exist $\delta=\delta_b\in(0,1)$ and an increasing function $h=h_b:(0,\infty)\to(0,\infty)$, with
$h(s)<\frac{s}{2}$ and $\lim\limits_{s\to\infty}h(s)=\infty$, such that the asymptotic \eqref{eq:surprise} holds, and, cf.~\eqref{eq:S0},
\begin{equation}\label{eq:Sdelta}
  \lim_{s\to\infty} b\bigl(h(s)\bigr)s^{1+\delta}=0.
\end{equation}
For any $n\in\N$, we denote by $\S[n]$ the subclass of functions $b$ from $\S$ such that
\begin{equation}\label{eq:finiteintd}
        \int_{-\infty}^\rho b(s) \,ds+\int_\rho^\infty b(s) s^{n-1}\,ds<\infty.
    \end{equation}
\end{defn}

\begin{rem}
It is worth noting again that, for a tail-decreasing function, \eqref{eq:surprise} implies that $b$ is long-tailed.
\end{rem}

\begin{rem}\label{rem:tosummarise}
  By Lemma~\ref{le:SimplsubexpR+}, any function $b\in\S$ is weakly sub-exponential on $\R$. Moreover, by~Lemma~\ref{lem:repeatFoss}, any function $b\in\S[1]$ is sub-exponential on $\R$.
\end{rem}

\begin{rem}
Let $b\in\S$, and $s_0>0$ be such that $h(2s_0)>\rho$. Then the monotonicity of $b$ and $h$ implies $b(s)\leq b(h(2s))$, $s>s_0$; and hence, because of \eqref{eq:Sdelta}, for $B:=2^{-1-\delta}$, there exists $s_1\geq s_0$, such that
\begin{equation}\label{eq:Eimpliespolbdd}
  b(s)\leq Bs^{-1-\delta}, \quad s\geq s_1.
\end{equation}
\end{rem}

Below we will show that $\S$ and $\S[n]$, $n\in\N$ are closed under some simple transformations of functions. For an arbitrary function $b\in\S$, we consider the following transformed functions:
1)~for fixed $p>0$, $q>0$, $r\in\R$, we set
  \begin{equation}\label{eq:aff}
    \widetilde{b}(s):=p b(q s+r), \quad s\in\R;
  \end{equation}
2) for a fixed $s_0>0$ and a fixed bounded function $c:\R\to\R_+$, we set
  \begin{equation}\label{eq:cutbyc}
    \breve{b}(s):=\1_{(-\infty,s_0)}(s) c(s)+\1_{[s_0,\infty)}(s)b(s), \quad s\in\R;
  \end{equation}
3) for any $\alpha\in(0,1]$, we denote
  \[
    b_\alpha(s):=\bigl(b(s)\bigr)^\alpha, \quad s\in\R.
  \] 

\begin{thm}\label{thm:coolthm}
  \begin{enumerate}
    \item Let $b\in\S$. Then the functions $\widetilde{b}$ and $\breve{b}$ defined in \eqref{eq:aff} and \eqref{eq:cutbyc}, correspondingly, also belong to $\S$ for all admissible values of their parameters. If, additionally, there exists $\alpha'\in(0,1)$ such that $b_{\alpha'}\in L^1(\R_+)$, then there exists $\alpha_0\in(\alpha',1)$, such that $b_\alpha\in\S$ for all $\alpha\in[\alpha_0,1]$.
    \item Let $b\in\S[n]$ for some $n\in\N$. Then $\widetilde{b}\in\S[n]$. If, additionally, the function $c$ in \eqref{eq:cutbyc} is integrable on $(-\infty,s_0)$, then $\breve{b}\in\S[n]$. Finally, if there exists $\alpha'\in(0,1)$ such that \eqref{eq:finiteintd} holds for $b=b_{\alpha'}$, then there exists $\alpha_0\in(\alpha',1)$, such that $b_\alpha\in\S[n]$ for all $\alpha\in[\alpha_0,1]$. Moreover, in the latter case, there exist $B_0>0$ and $\rho_0>0$, such that, for all $\alpha\in(\alpha_0,1]$,
         \begin{equation}\label{eq:ineqinsteadsubexp}
            \int_\R \bigl(b(s-\tau)\bigr)^\alpha \bigl(b(\tau)\bigr)^\alpha\,d\tau \leq B_0 \bigl(b(s)\bigr)^\alpha, \quad s\geq \rho_0,
        \end{equation}
  \end{enumerate}
 \end{thm}
\begin{proof}
It is very straightforward to check that if $b$ is long-tailed, tail-decreasing and tail-log-convex, then $\widetilde{b},\breve{b},b_\alpha$  also have these properties for all admissible values of their parameters. Let $h:(0,\infty)\to(0,\infty)$ be such that
$h(s)<\frac{s}{2}$, $\lim\limits_{s\to\infty}h(s)=\infty$, and \eqref{eq:surprise} hold. Let also \eqref{eq:Sdelta} hold for some $\delta>0$.

(i) Evidently, both \eqref{eq:surprise} and \eqref{eq:Sdelta} hold, with $b$ replaced by $\breve{b}$. Next, $\breve{b}\in L^1(\R_+)$ and $\breve{b}$ is bounded. Hence $\breve{b}\in\S$. If $b\in\S[n]$ and $c$ is integrable on $(-\infty,s_0)$, then \eqref{eq:finiteintd} holds for $b$ replaced by $\breve{b}$. Cf.~also Corollary~\ref{cor:right-side_is_right-side}.

(ii) Set, for the given $q>0$, $r\in\R$, $\widetilde{h}(s):=\frac{1}{q} h(qs+r)-\frac{r}{2q}\1_{\R_+}(r)$,  $s\in[s_1,\infty)$, where $s_1>0$ is such that $qs_1+r>0$ and $h(qs+r)>\frac{r}{2q}$ for all $s\geq s_1$, and $\widetilde{h}$ is increasing on $(0,s_1)$, such that $\widetilde{h}(s)<\min\bigl\{\frac{s}{2},\widetilde{h}(s_1)\bigr\}$, $s\in(0,s_1)$.
By~Proposition~\ref{prop:useful}, \eqref{eq:surprise} is equivalent to  \eqref{eq:uniformlongtailedR}. Then, by~\eqref{eq:aff}, we have, 
$\sup_{|\tau|\leq \widetilde{h}(s)}\Bigl\lvert \frac{\widetilde{b}(s+\tau)}{\widetilde{b}(s)}-1\Bigr\rvert\to0,$
as $s\to\infty$.
Therefore, again by Proposition~\ref{prop:useful}, \eqref{eq:surprise} holds for $b$ replaced by $\widetilde{b}$.
Next, it is straightforward to get from \eqref{eq:longtaileddef}, \eqref{eq:Sdelta}, that $\widetilde{b}\bigl(\widetilde{h}(s)\bigr) s^{1+\delta} \to 0$, as $s\to\infty$.
Therefore, $\widetilde{b}\in\S$.
Finally, $b\in\S[n]$ for some $n\in\N$, trivially implies $\widetilde{b}\in\S[n]$.

(iii) Evidently, the convergence \eqref{eq:surprise} implies the same one with $b$ replaced by $b_\alpha$, with the same $h$ and for any $\alpha\in(0,1)$.
Next, let $\alpha'\in(0,1)$ be such that $b_{\alpha'}\in L^1(\R_+)$. By the well-known log-convexity of $L^p$-norms (for $p>0$), for any $\alpha\in(\alpha',1)$ and for $\beta:=\frac{\alpha-\alpha'}{\alpha(1-\alpha')}\in(0,1)$, we have $\frac{1}{\alpha}=\frac{1-\beta}{\alpha'}+\beta$ and
\begin{equation}\label{eq:normlogconv}
\lVert b\rVert_{L^\alpha(\R_+)} \leq
\lVert b\rVert_{L^{\alpha'}(\R_+)} ^ {1-\beta}\, \lVert b\rVert_{L^1(\R_+)}^\beta<\infty,
\end{equation}
i.e.~$b_\alpha\in L^1(\R_+)$ for all $\alpha\in(\alpha',1)$.
Take and fix an $\alpha_0\in\Bigl(\max\Bigl\{\alpha',\frac{1}{1+\delta}\Bigr\},1\Bigr)$. Then, for any $\alpha\in
[\alpha_0,1]$, we have that $\delta':=\alpha(1+\delta)-1\in(0,\delta]$, and hence, by~\eqref{eq:Sdelta},
  \[
  \lim_{s\to\infty} b_\alpha\bigl(h(s)\bigr) s^{1+\delta'} =
  \lim_{s\to\infty} \Bigl( b\bigl(h(s)\bigr) s^{1+\delta}\Bigr)^\alpha=0.
  \]
Therefore, $b_\alpha\in\S$, $\alpha\in[\alpha_0,1]$.

Let, additionally, \eqref{eq:finiteintd} hold for both $b$ and $b_{\alpha'}$ (i.e., in particular, $b\in\S[n]$) and for some $n\in\N$. Then one can use again the log-convexity of $L^p$-norms, now for $L^p\bigl((\rho,\infty),s^n\,ds\bigr)$ spaces, to deduce that $b_\alpha\in\S[n]$, $\alpha\in[\alpha_0,1]$.

Finally, $b,b_{\alpha_0}\in\S[n]$, $n\in\N$, implies $b,b_{\alpha_0}\in\S[1]$, and hence, cf.~Remark~\ref{rem:tosummarise}, $b$ and $b_{\alpha_0}$ are sub-exponential on $\R$, i.e.~\eqref{eq:realsubexponR} holds for both $b$ and $b_{\alpha_0}$. Therefore, for an arbitrary $\eps\in(0,1)$, there exists $\rho_0=\rho_0(\eps, b,b_{\alpha_0})>\rho $ (where $\rho$ is from Definition~\ref{def:super-subexp}) and
  $B_0:=2(1+\eps)\max\bigl\{\int_\R b(s)\,ds,\int_\R b_{\alpha_0}(s)\,ds\bigr\}>0$,
such that, for all $s\geq\rho_0$,
\begin{equation}\label{eq:weknow}
\int_\R b(s-\tau)b(\tau)\,d\tau\leq B_0 b(s),\qquad
\int_\R b_{\alpha_0}(s-\tau)b_{\alpha_0}(\tau)\,d\tau\leq B_0 b_{\alpha_0}(s).
\end{equation}
Then, applying again the norm log-convexity arguments, cf.~\eqref{eq:normlogconv}, one gets, for any fixed $s\geq\rho_0$ and for all $\alpha\in(\alpha_0,1)$
\[
\int_\R \bigl(b(s-\tau)b(\tau)\bigr)^\alpha d\tau\leq
  \biggl(\int_\R \bigl(b(s-\tau)b(\tau)\bigr)^{\alpha_0} d\tau\biggr)^{\frac{1}{\alpha_0}(1-\beta)\alpha}
  \biggl(\int_\R b(s-\tau)b(\tau) d\tau\biggr)^{\beta\alpha },
\]
where $\beta=\frac{\alpha-\alpha_0}{\alpha(1-\alpha_0)}\in(0,1)$. Combining the latter inequality with \eqref{eq:weknow}, one gets
\[
\int_\R \bigl(b(s-\tau)b(\tau)\bigr)^\alpha d\tau\leq
\bigl(B_0 (b(s))^{\alpha_0}\bigr)^{\frac{1}{\alpha_0}(1-\beta)\alpha}
\bigl(B_0 b(s)\bigr)^{\beta\alpha}=B_0 \bigl(b(s)\bigr)^\alpha.
\]
The theorem is fully proved now.
\end{proof}

It is naturally to expect that asymptotically small changes of  tail properties preserve the sub-exponential property of a function. Namely, consider the following definition.
\begin{defn}
  Two functions $b_1,b_2:\R\to\R_+$ are said to be {\em weakly tail-equivalent} if
  \begin{equation*}
    0<\liminf_{s\to\infty}\frac{b_1(s)}{b_2(s)}\leq \limsup_{s\to\infty}\frac{b_1(s)}{b_2(s)}<\infty,
  \end{equation*}
  or, in other words, if there exist $\rho>0$ and $C_2\geq C_1>0$, such that,
  \begin{equation}\label{eq:ineqweakeqv}
    C_1 b_1(s)\leq b_2(s)\leq C_2 b_1(s), \quad s\geq\rho.
  \end{equation}
\end{defn}

The proof of the following statement is a straightforward consequence of \cite[Theorem 4.8]{FKZ2013}.
\begin{prop}
  Let $b_1:\R\to\R_+$ be a weakly sub-exponential on $\R$ function. Let $b_2:\R\to\R_+$ be a long-tailed function which is weakly tail-equivalent to $b_1$. Then $b_2$ is weakly sub-exponential on $\R$ as well. If, additionally, \eqref{eq:quasi-decreasing} holds for $b=b_1$, then $b_2$ is sub-exponential on $\R$.
\end{prop}

\begin{prop}\label{prop:tailprop}
  Let $b_1\in\S$ and let $b_2:\R\to\R_+$ be a bounded tail-decreasing and tail-log-convex function, such that
  \begin{equation}\label{eq:tailprop}
    \lim_{s\to\infty}\frac{b_2(s)}{b_1(s)}=C\in(0,\infty).
  \end{equation}
  Then $b_2\in\S$.
\end{prop}
\begin{proof}
Clearly, $b_2$ is long-tailed as $b_1$ is such. Let  $\delta\in(0,1)$ and $h$ be an increasing function, with
$h(s)<\frac{s}{2}$, $\lim\limits_{s\to\infty}h(s)=\infty$, such that  \eqref{eq:surprise} and \eqref{eq:Sdelta} hold for $b=b_1$.
Let $\eps\in\bigl(0,\min\{1,C\}\bigr)$, and choose $\rho>1$, such that $b_2$ is decreasing and log-convex on $[\rho,\infty)$, and $b_2(\rho)\leq 1$. By~\eqref{eq:tailprop} and \eqref{eq:surprise} (for $b=b_1$), there exists $\rho_1\geq\rho$, such that
  \begin{equation}\label{eq:closeC12}
0<(C-\eps)b_1(s)\leq b_2(s)\leq (C+\eps) b_1(s), \qquad \quad
\biggl\lvert \frac{b_1(s\pm h(s))}{b_1(s)}-1\biggr\rvert<\eps
  \end{equation}
  for all $s\geq \rho_1$.
Since $b_2$ is bounded and $b_1\in L^1(\R_+)$, we have from \eqref{eq:closeC12} that $b_2\in L^1(\R_+)$.
By \eqref{eq:closeC12}, for any $s\geq 2 \rho_1$, 
\[
\biggl\lvert \frac{b_2(s\pm h(s))}{b_2(s)}-1\biggr\rvert
<\max\biggl\{\eps\frac{C+\eps}{C-\eps}+\frac{C+\eps}{C-\eps}-1,\
\eps\frac{C-\eps}{C+\eps}+1-\frac{C-\eps}{C+\eps}\biggr\}.
\]
Since the latter expression may be arbitrary small, by an appropriate choice of $\eps$, one gets that \eqref{eq:surprise} holds for $b=b_2$. Finally, \eqref{eq:Sdelta} with $b=b_1$ and \eqref{eq:tailprop} imply that \eqref{eq:Sdelta} holds with $b=b_2$ and the same $\delta$ and $h$.
\end{proof}

\begin{rem}\label{rem:basbsadfsaofas}
In the assumptions of the previous theorem, if, additionally, $b_1\in\S[n]$ for some $n\in\N$, and $b_2$ is integrable on $(-\infty,-\rho_2)$ for some $\rho_2>0$, then $b_2\in\S[n]$ (because of \eqref{eq:closeC12} and the boundedness of $b_2$).
\end{rem}

On the other hand, if one can check that both functions $b_1$ and $b_2$ satisfy \eqref{eq:surprise} with the same function $h(s)$, then
the sufficient condition to verify \eqref{eq:Sdelta} for $b=b_2$, provided that it holds for $b=b_1$, is much weaker than
\eqref{eq:tailprop}. To present the corresponding statement, consider the following definition.

\begin{defn}\label{def:logeqv}
Let $b_1,b_2:\R\to\R_+$ and, for some $\rho\geq0$, $b_i(s)>0$ for all $s\in[\rho,\infty)$, $i=1,2$. The functions $b_1$ and $b_2$ are said to be  {\em (asymptotically) log-equivalent}, if
 \begin{equation}\label{eq:log-equiv}
   \log b_1(s) \sim \log b_2(s), \quad s\to\infty.
 \end{equation}
\end{defn}

\begin{prop}\label{prop:tailprop2cool}
  Let $b_1\in\S$ and let $h$ be the function corresponding to Definition~\ref{def:super-subexp} with $b=b_1$. 
  Let $b_2:\R\to\R_+$ be a bounded tail-decreasing and tail-log-convex function, such that \eqref{eq:surprise} holds with $b=b_2$  and the same $h$. Suppose that $b_1$ and $b_2$ are log-equivalent.  Then $b_2\in\S$. If, additionally, there exists $\alpha'\in(0,1)$, such that \eqref{eq:finiteintd} holds with $b=(b_1)^{\alpha'}$  and $b_2$ is integrable on $(-\infty,\rho)$, then $b_2\in\S[n]$.
\end{prop}
\begin{proof}
Let $\delta\in(0,1)$ be such that \eqref{eq:Sdelta} holds for $b$ replaced by $b_1$. Take an arbitrary $\eps\in\bigl(0,\frac{\delta}{1+\delta}\bigr)$. By \eqref{eq:log-equiv}, there exists $\rho_\eps>0$, such that $b_i(s)<1$,  $s>\rho_\eps$, $i=1,2$, and 
\begin{gather}
-(1-\eps) \log b_1(s)\leq -\log b_2 (s)\leq -(1+\eps)\log b_1(s), \quad s>\rho_\eps,\notag\\
 b_1(s)^{1+\eps}\leq b_2(s)\leq b_1(s)^{1-\eps}, \quad s>\rho_\eps.\label{eq:asfdsafsadfdsewtewtegw}
\end{gather}
Since $h(s)\to\infty$, $s\to\infty$, there exists $\rho_0>\rho_\eps$, such that $h(s)>\rho_\eps$ for any $s>\rho_0$. Then, by \eqref{eq:asfdsafsadfdsewtewtegw}, we have, for all $s>\rho_0$,
\begin{equation*}
b_2(h(s))s^{(1+\delta)(1-\eps)}<b_1(h(s))^{1-\eps}s^{(1+\delta)(1-\eps)}
=\bigr(b_1(h(s))s^{1+\delta}\bigr)^{1-\eps},
\end{equation*}
and therefore, \eqref{eq:Sdelta} holds with $b=b_2$ and $\delta$ replaced by 
$(1+\delta)(1-\eps)-1=\delta-\eps(1+\delta)\in(0,1)$,
that proves the first statement. To prove the second one, assume, additionally, that $\eps<1-\alpha'$. Then, by~\eqref{eq:asfdsafsadfdsewtewtegw}, we have
$b_2(s)s^{n-1}\leq b_1(s)^{1-\eps} s^{n-1}<b_1(s)^{\alpha'}s^{n-1}$ for all $s>\rho_\eps$,
as $b_1(s)<1$ here. 
\end{proof}

\subsection{Examples}\label{subsec:examples}

We consider now examples of functions $b\in\S$. Because of Proposition~\ref{prop:tailprop2cool}, we will classify these functions `up to log-equivalence', i.e.~by tail properties of the function
\[
l(s):=-\log b(s).
\]

Taking into account the result of Theorem~\ref{thm:coolthm} concerning the function $\breve{b}$, it will be enough to define $b$ on some $(s_0,\infty)$, $s_0>0$ only. 
Next, by~Lemma~\ref{le:SimplsubexpR+}, the function $b_+$ defined by \eqref{eq:defofbplus} is a sub-exponential density on $\R_+$. Therefore, one can use the classical examples of such densities, see e.g.~\cite{FKZ2013}. However, using the result of Theorem~\ref{thm:coolthm} concerning the function $\widetilde{b}$, one can consider that examples in their `simplest' forms (ignoring any shifts of the argument or scales of the argument or the function itself).

Now we consider different asymptotics of the function $l(s)=-\log b(s)$. In~all particular examples below, it is straightforward to check that each particular bounded functions $b$ is such that $b'(s)<0$ and $(\log b(s))''>0$ for all big enough values of $s$, i.e.~$b$ is tail-decreasing and tail-log-convex.

\subsection*{Class 1: $l(s)\sim D\log s$, $s\to\infty$, $D>1$}

\paragraph{Polynomial decay.} Let $b:\R\to\R_+$ be a bounded tail-decreasing tail-log-convex function, such that $b(s)\sim q s^{-D}$, $s\to\infty$, $D>1$, $q>0$.
By Proposition~\ref{prop:tailprop}, to show that $b\in\S$, it is enough to prove this for 
\[
b(s)=\1_{\R_+}(s)(1+s)^{-D}, \quad s\in\R.
\]
For an arbitrary $\ga\in\bigl(\frac{1}{D},1\bigr)$, consider $h(s)=s^\ga$, $s>0$. Then it is straightforward to check that \eqref{eq:surprise} and \eqref{eq:Sdelta} hold, provided that $\delta\in(0, \ga D-1)\subset(0,1)$. As a result, $b\in\S$. Clearly, $b\in\S[n]$ for $D>n$, cf.~Remark~\ref{rem:basbsadfsaofas}. 

Classical examples of the polynomially decaying probability densities in \cite{FKZ2013} can be described by the following functions: 
  \begin{enumerate}
  \item {\em Student's $t$-function}: $\mathscr{T}(s)=(1+s^2)^{-p}$, $s>0$, $p>\frac{1}{2}$.
  Note that $\mathscr{T}\in\S[n]$, $n\in\N$, if only $p>\frac{n}{2}$.
 The case $p=1$ is referred to the Cauchy distribution, the corresponding function belongs to $\S[n]$ for $n=1$ only.
    \item {\em The L\'evy function}: 
    $\mathscr{L}(s)=s^{-\frac{3}{2}}\exp\bigl(-\frac{c}{s}\bigr)$, $s>0$, $c>0$.
    \item {\em The Burr function}: $\mathscr{B}(s)= s^{c-1}(1+s^c)^{-k-1}$, $s>0$, $c>0$, $k>0$.
    Note that the case $c=1$ is related to the Pareto distribution; the latter has the density $kp^k \mathscr{B}(s-1)\1_{[p,\infty)}(s)$ for any $p>0$.
\end{enumerate}

\paragraph{Logarithmic perturbation of the polynomial decay.} 
Let $D>1$, $\nu\in\R$, and
\[
b(s)=\1_{(1,\infty)}(s)(\log s)^\nu s^{-D}, \quad s\in\R.
\]
We are going to apply Proposition~\ref{prop:tailprop2cool} now, with $b_1(s)=s^{-D}$ and $b_2(s)=(\log s)^\nu s^{-D}$.
Indeed, then \eqref{eq:log-equiv} evidently holds. It remains to check that \eqref{eq:surprise} holds for both $b_1$ and $b_2$ with the same $h(s)=s^\ga$, $\ga\in(0,1)$. One has then
$\frac{\log (s\pm s^\ga)}{\log s} \to 1$ as $s\to\infty$, that yields the needed. 
  
\subsection*{Class 2:  $l(s)\sim D(\log s)^q$, $s\to\infty$, $q>1$, $D>0$} 
\label{subsubsect:lognorm}

Consider the function
\[
N(s):=\1_{\R_+}(s)\exp\bigl(-D(\log s)^q\bigr), \quad s\in\R.
\]
Take $h(s)=\1_{[\rho,\infty)}(s)s^{\frac{1}{q}}$, where $\rho>1$ is chosen such that $h(s)<\frac{s}{2}$ for $s\geq\rho$. Since $q>1$, we have
\begin{align*}
\frac{N(s\pm h(s))}{N(s)}
    &=
    \exp\biggl\{D(\log s)^q\biggl(1-\Bigl(1+\frac{\log(1\pm s^{\frac{1}{q}-1})}{\log s}\Bigr)^q\biggr)\biggr\}\\
&\sim \exp\biggl\{D(\log s)^q\biggl(\mp q \frac{s^{\frac{1}{q}-1}}{\log s}\biggr)\biggr\}\sim 1, \quad s\to\infty,
\end{align*}
that proves \eqref{eq:surprise}.
Next, for any $\delta\in\R$,
    \[
    N\bigl(s^{\frac{1}{q}}\bigr)s^{1+\delta}=
    \exp\bigl(-Dq^{-q}(\log s)^q+(1+\delta)\log s\bigr)\to0,\quad s\to\infty.
    \]
    As a result, $N\in\S$. Moreover, evidently, $N\in\S[n]$, for any $n\in\N$.
 
    We may also consider Proposition~\ref{prop:tailprop2cool} for $b_1=b$ and $b_2=pb$, where $b_2$ is tail-decreasing and tail-log-convex function, such that $\log p = o(\log b)$ (that is equivalent to $\log b_1\sim\log b_2$) and $p$ satisfies \eqref{eq:surprise} with $h(s)=s^{\frac{1}{q}}$. According to the results above, a natural example of such $p(s)$ might be $s^D$, $D\in\R$. It is straightforward to verify that, for any $D\in\R$, $b_2=pb_1$ is tail-decreasing and tail-log-convex. As a result, then $b_2\in\S[n]$, $n\in\N$.

The classical log-normal distribution has the density described by the function
    $\mathscr{N}(s)=\frac{1}{s}\exp\Bigl(-\frac{(\log s)^2}{2\gamma^2}\Bigr)$, $s>0$, $\gamma>0$, that can be an example of the function $b_2$ above.

\subsection*{Class 3: $l(s)\sim s^\alpha$, $\alpha\in(0,1)$}

    Consider, for any $\alpha\in(0,1)$, the so-called {\em fractional exponent}
    \begin{equation}
        w(s)=\1_{\R_+}(s)\exp(-s^\alpha), \quad s\in\R.\label{eq:fractionalexp}
    \end{equation}
    Set $h(s)=\1_{[\rho,\infty)}(s)(\log s)^{\frac{2}{\alpha}}$, where $\rho>0$ is chosen such that $h(s)<\frac{s}{2}$ for $s\geq\rho$. Then $h(s)=o(s)$, $s\to\infty$, and hence
\[
    \frac{w\bigl(s\pm h(s)\bigr)}{w(s)}=
    \exp\biggl\{-s^\alpha\biggl( \Bigl( 1\pm \frac{h(s)}{s}\Bigr)^\alpha-1 \biggr)\biggr\}\sim
    \exp\biggl\{- s^\alpha\biggl( \pm \alpha \frac{h(s)}{s} \biggr)\biggr\}
    \sim 1, 
\]
as $s\to\infty$, that proves \eqref{eq:surprise}. Next, for any $\delta\in\R$,
   \[
   w\bigl(h(s)\bigr)s^{1+\delta}=\exp\bigl(-(\log s)^2+(1+\delta)\log s\bigr)\to0, \quad s\to\infty.
   \]
    As a result, $w\in\S$. It is clear also that $w\in\S[n]$ for all $n\in\N$. 

Similarly to the above, one can show that $pw\in\S$, provided that, in particular, $\log p=o(\log w)$ and \eqref{eq:surprise} holds for $b=p$ and $h(s)=(\log s)^{\frac{2}{\alpha}}$. Again, one can consider $p(s)=s^D$, $D\in\R$, since it satisfies \eqref{eq:surprise} with $h(s)=s^\ga>(\log s)^{\frac{2}{\alpha}}$, $\alpha,\gamma\in(0,1)$, and big enough $s$. As before, the verification that, for any $D\in\R$, $b_2=pb_1$ is tail-decreasing and tail-log-convex is straightforward. 

The probability density of the classical Weibull distribution is described by the function
    $\mathscr{W}(s)= s^{\alpha-1}\exp(-s^\alpha)$, $s\geq\rho>0$, $\alpha\in(0,1)$. 
 By the above, $\mathscr{W}\in\S[n]$, $n\in\N$.   Note that $\int_s^\infty\mathscr{W}(\tau)\,d\tau=\frac{1}{\alpha} w(s)$, where $w$ is given by \eqref{eq:fractionalexp}.

\subsection*{Class 4: $l(s)\sim \frac{s}{(\log s)^q}$, $q>1$}

    Consider also a function which decays `slightly' slowly than an exponential function. Namely, let, for an arbitrary fixed $q>1$,
   \[
        g(s)=\1_{\R_+}(s)\exp\Bigl(-\frac{s}{(\log s)^q}\Bigr),\quad s\in\R.
   \]
   Take, for an arbitrary $\gamma\in(1,q)$, $h(s)=(\log s)^\gamma$, $s>0$; and denote, for a brevity, $p(s):=\frac{h(s)}{s}\to0$, $s\to\infty$. Then, $\log (s+h(s))=\log s+\log(1+p(s))$. Set also
   $r(s):=\frac{\log(1+p(s))}{\log s}\to0$, $s\to\infty$.
   Then, for any $s>e^{q+1}$, we have
   \begin{align*}
   &\quad\log\frac{g(s+ h(s))}{g(s)}
   =
 \frac{s}{(\log s)^q}\Bigl(1-\frac{1+p(s)}{\bigl(1+r(s)\bigr)^q}\Bigr)\\
   &=
   \frac{1}{\bigl(1+r(s)\bigr)^q}\biggl(q\frac{\bigl(1+r(s)\bigr)^q  -1}{q r(s)} \frac{\log(1+p(s))}{p(s)}(\log s)^{\gamma-q-1}-(\log s)^{\gamma-q}\biggr)\to0,
   \end{align*}
   as $s\to\infty$, since $\gamma<q$; and similarly $\log\frac{g(s- h(s))}{g(s)}\to0$, $s\to\infty$. Therefore, \eqref{eq:surprise} holds for $b=g$.
   Next,
   \[
   \log \bigl(g(h(s))s^{1+\delta}\bigr)=
   -(\log s)\Bigl(\frac{(\log s)^{\gamma-1}}{\gamma^q(\log \log s)^q}
   -(1+\delta)\Bigr)\to-\infty, \quad s\to\infty,
   \]
   that yields \eqref{eq:Sdelta} for $b=g$. As a result, $g\in\S$. Again, evidently, $g\in\S[n]$, $n\in\N$. The same arguments as before show that, for any $D\in\R$, the function $s^D g(s)$ belongs to $\S[n]$ as well.

    \begin{rem}\label{rem:exception}
      Naturally, $q\in(0,1]$ gives behavior of $g(s)$ more `close' to the exponential function. Unfortunately, our approach does not cover this case: the analysis above shows that $h(s)$, to fulfill even \eqref{eq:S0}, must grow faster than $\log s$, whereas so `big' $h(s)$ would not fulfill \eqref{eq:surprise}. In general, Lemma~\ref{le:SimplsubexpR+} gives a sufficient condition only, to get a sub-exponential density on $\R_+$. It can be shown, see e.g.~\cite[Example~1.4.3]{EKM1997}, that a probability distribution, whose density $b$ on $\R_+$ is such that $\int_s^\infty b(\tau)\,d\tau\sim g(s)$, $s\to\infty$, with $q>0$, is a sub-exponential distribution (for the latter definition, see e.g.~\cite[Definition~3.1]{FKZ2013}). Then we expect that $b(s)\sim -g'(s)$, $s\to\infty$, and it is easy to see that $\log (-g'(s))\sim \log g(s)$, $s\to\infty$. It should be stressed though that, in general, sub-exponential property of a distribution does not imply the corresponding property of its density, cf.~\cite[Section~4.2]{FKZ2013}. Therefore, we can not state that the function $b$ above is a sub-exponential one for $q\in(0,1]$.
    \end{rem}
   
Combining the results above, one gets the following statement.
\begin{cor}
Let $b:\R\to\R_+$ be a bounded tail-decreasing and tail-log-convex function, such that, for some $C>0$, the function $Cb(s)$ has either of the following asymptotics as $s\to\infty$ 
\begin{alignat*}{2}
&(\log s)^\mu s^{-(n+\delta)},\qquad   && (\log s)^\mu  s^\nu \exp\bigl(-D(\log s)^q\bigr), \\
&(\log s)^\mu s^\nu \exp\bigl(-s^\alpha\bigr), \qquad  &&   
(\log s)^\mu s^\nu \exp\Bigl(-\frac{s}{(\log s)^q}\Bigr),
\end{alignat*}
where $D, \delta>0$, $q>1$, $\alpha\in(0,1)$, $\nu,\mu\in\R$. Then $b\in\S[n]$, $n\in\N$.
\end{cor}

\section{Analogues of Kesten's bound on $\X$}\label{sec:KbddRd}

We start with a simple corollary of Propositions~\ref{prop:quasiKersten} and~\ref{prop:ba:suff_cond1}.

\begin{prop}\label{prop:coolsasd}
Let a function $\upomega:\X\to(0,+\infty)$ be such that \eqref{eq:defofsetOmegala} holds, and
   \begin{equation} \label{eq:aconvomegaoveromega}
\limsup_{\la\to0+}\sup_{x\in\Omega_\la}\frac{(a*\upomega)(x)}{\upomega(x)}\leq1.
\end{equation}
Let also $a\in L^\infty(\X)$ and $\lVert a\rVert_\upomega<\infty$. 
Then, for any $\delta\in(0,1)$, there exist $c_\delta>0$ and $\la=\la(\delta)\in(0,1)$, such that 
\begin{equation}\label{eq:new2}
a^{*n}(x)\leq c_\delta (1+\delta)^{n-1} \min\bigl\{\la,\upomega(x)\bigr\}, \quad x\in\X.
\end{equation}
\end{prop}
\begin{proof} Take any $\delta\in(0,1)$. By Proposition~\ref{prop:ba:suff_cond1}, there exists $\la=\la(\delta,\upomega)\in(0,1)$, such that  \eqref{eq:ba:Phi_est_suff_cond} holds, with $ \phi$ given by \eqref{eq:defofomegala} and $\ga=1+\delta$. 
   Denote  $\|a\|_\infty:=\|a\|_{L^\infty(\X)}$. We have
   \begin{equation}\label{eq:uoiomegalambda}
  \frac{a(x)}{\upomega_\la(x)}\leq \frac{\|a\|_\infty}{\la}\1_{\X\setminus\Omega_\la}(x)+
\frac{a(x)}{\upomega(x)}\1_{\Omega_\la}(x)
\leq \frac{\|a\|_{\infty}}{\la}+\lVert a\rVert_{\upomega}=:c_\delta<\infty,
\end{equation}
and one can apply Proposition~\ref{prop:quasiKersten} that yields the statement.
\end{proof}

\begin{rem}
Note that, for $d=1$, Proposition~\ref{prop:almostdream} implies that, if only $\upomega\in\S[1]$, $a(s)=o(\upomega(s))$, $s\to\infty$, and \eqref{eq:normed} holds, then $(a*\upomega)(s)\sim \upomega(s)$, $s\to\infty$; and then, in particular, \eqref{eq:aconvomegaoveromega} holds. Next, in~the course of the proof of Theorem~\ref{thm:KestenonR} (still for $d=1$), we slightly  weakened the restriction $a=o(\upomega)$ for the case where $a$ itself is sub-exponential, by setting $\upomega:=a^{*m}$ with large enough $m\in\N$, since then $\frac{a}{\upomega}\sim \frac{1}{m}$ may be chosen arbitrary small. As it was mentioned in the Introduction, for the multi-dimensional case, we do not have a theory of sub-exponential densities. Therefore, we consider more `rough' candidate for $\upomega$ to ensure \eqref{eq:aconvomegaoveromega}, namely, $\upomega(x)=a(x)^\alpha$ for $\alpha\in(0,1)$; then, in particular, $a(x)=o(\upomega(x))$, $|x|\to\infty$ if, for example, $a(x)=b(|x|)$, $x\in\X$, with a tail-decreasing function $b$, cf. Definition~\ref{def:Dtclass} below. In particular, the results of this Section, following from \eqref{eq:new2}, yield upper bounds for $a^{*n}(x)$ with the right-hand side heavier than $a(x)$ at infinity.     
\end{rem}

\begin{defn}\label{def:Dtclass}
Let $\Dt$ be the set of all bounded functions $b:\R\to(0,\infty)$, such that $b$ is tail-decreasing (cf.~Definition~\ref{def:taildecreasing}) and
      $\int_0^\infty b(s) s^{d-1}\,ds<\infty$.
Let $\D\subset\Dt$ denote the subset of all functions from $\Dt$ which are (strictly) decreasing to $0$ on $\R_+$.
\end{defn}

\begin{rem}\label{rem:saddssdadsasadsda1}
It is easy to see that, if $b^{\alpha_0}\in\D$ for some $\alpha_0\in(0,1)$, then $b^\alpha\in\D$ for all $\alpha\in[\alpha_0,1]$.
\end{rem}

Consider an analogue of long-tailed functions on $\X$.
\begin{lem}\label{le:rightsidelongtailedmeansspatially}
Let $b:\R_+\to\R_+$ be a long-tailed function (cf.~Definition~\ref{def:longtailedfunc}), and $c(x)=b(|x|)$, $x\in\X$.  Then, for any $r>0$,
  \begin{equation}\label{eq:long-tailed}
        \lim_{|x|\to\infty} \sup_{|y|\leq r} \biggl\lvert \frac{c(x+y)}{c(x)} - 1 \biggr\rvert = 0.
  \end{equation}
\end{lem}
\begin{proof}
Evidently, $|y|\leq r$ implies that $h:=|x+y|-|x|\in[-r,r]$ for each $x\in\X$. Therefore,
\[
\sup_{|y|\leq r} \biggl\lvert \frac{b(|x+y|)}{b(|x|)}-1\biggr\rvert \leq \sup_{h\in[-r,r]} \biggl\lvert \frac{b(|x|+h)}{b(|x|)}-1\biggr\rvert \to 0, \quad |x|\to\infty,
\]
because of \eqref{eq:longtaileddef-uniform}.
\end{proof}

We will assume in the sequel, that $a$ is bounded by a radially symmetric function:
\begin{equation}\label{ass:newA5}
     \text{There exists $b^+\in\D$, such that } 
     a(x)\leq b^+(|x|), \ \text{for a.a.~} x\in\X.
\end{equation}

We start with the following sufficient condition.

\begin{prop}\label{prop:poldecbdd}
Let \eqref{ass:newA5} hold with $b^+\in\D$ which is log-equivalent, cf.~Definition~\ref{def:logeqv}, to the function $b$, given by 
\begin{equation}\label{eq:polynomialb}
b(s):=\1_{\R_+}(s)\frac{M}{(1+s)^{d+\mu }}, \quad s\in\R,
\end{equation}
for some $\mu,M>0$.
Then there exists $\alpha_0\in(0,1)$, such that, for all $\alpha\in(\alpha_0,1)$,
 the function $\upomega(x)=b(|x|)^\alpha$, $x\in\X$, satisfies \eqref{eq:aconvomegaoveromega}.
 \end{prop}
\begin{proof} Set $\alpha_0:=\frac{d+\frac{\mu }{2}}{d+\mu }\in(0,1)$.
Take arbitrary $\alpha\in(\alpha_0,1)$ and $\epsilon\in(0,1-\alpha)$. 
Take also an arbitrary $\delta\in(0,1)$, and define $h(s)=s^\delta$, $s>0$. By \eqref{eq:asfdsafsadfdsewtewtegw}, applied to $b_1=b$ and $b_2=b^+$, there exists $s_\delta>2r$ such that, for all $ s>s_\delta$,
 \begin{equation}\label{eq:choiceofsdelta}
   h(s)<\frac{s}{2}, \qquad  b^+(s)\leq \bigl(b(s)\bigr)^{1-\epsilon}.
 \end{equation}
 For an arbitrary $x\in\X$ with $|x|>s_\delta$, we have a disjoint expansion $\X=D_1(x)\sqcup D_2(x)\sqcup D_3(x)$, where
\begin{gather*}
D_1(x):=\bigl\{|y|\leq h(|x|)\bigr\}, \quad D_2(x):=\Bigl\{h(|x|)<|y|\leq\frac{|x|}{2}\Bigr\},\quad
 D_3(x)=\Bigl\{|y|\geq \frac{|x|}{2}\Bigr\}.
\end{gather*}
Then, $\frac{(a*\upomega)(x)}{\upomega(x)}=I_1(x)+I_2(x)+I_3(x)$, where
\[
I_j(x):=\int_{D_j(x)}a(y)\frac{(1+|x|)^{(d+\mu )\alpha}}{(1+|x-y|)^{(d+\mu )\alpha}}dy, \quad j=1,2,3.
\]
Using the inequality $|x-y|\geq \bigl\lvert |x|-|y|\bigr\rvert$, $x,y\in\X$, one has that $|x-y|\geq |x|-|y|\geq |x|-|x|^\delta$ for $y\in D_1(x)$, $|x|>s_\delta$. Then
\begin{align*}
  I_1(x)\leq \biggl(\frac{1+|x|}{1+|x|-|x|^\delta}\biggr)^{(d+\mu )\alpha}\int_{D_1(x)} a (y)dy\to 1, \quad |x|\to\infty.
\end{align*}
Next, we  have, for any $|y|<\frac{|x|}{2}$, that $1+|x-y|\geq 1+|x|-|y|\geq\frac{1}{2}(1+|x|)$; therefore,
\begin{align*}
  I_2(x)\leq 2^{(d+\mu )\alpha} \int_{\{|y|>|x|^\delta\}} a (y)dy\to0, \quad |x|\to\infty.
\end{align*}
Finally, by \eqref{ass:newA5} and \eqref{eq:choiceofsdelta}, the inclusions $y\in D_3(x)$ and $|x|>s_\delta$ imply
$a(y)\leq b^+(|y|)\leq b(|y|)^{1-\epsilon}\leq b\bigl(\frac{|x|}{2}\bigr)^{1-\epsilon}$,
 and, therefore,
\begin{align*}
  I_3(x)&\leq
M\frac{(1+|x|)^{(d+\mu )\alpha}}{\bigl(1+\frac{|x|}{2}\bigr)^{(d+\mu)(1-\epsilon) }}
\int_{D_3(x)}\frac{1}{(1+|x-y|)^{(d+\mu )\alpha_0}}dy\\
&\leq M\frac{(1+|x|)^{(d+\mu )\alpha}}{\bigl(1+\frac{|x|}{2}\bigr)^{(d+\mu)(1-\epsilon) }}
\int_{\X}\frac{1}{(1+|y|)^{d+\frac{\mu }{2}}}dy\to0, \quad |x|\to\infty,
\end{align*}
as $1-\epsilon>\alpha$.
Since $b$ is decreasing on $\R_+$, we have, by \eqref{eq:defofsetOmegala}, that, for any $\la>0$, there exists $\rho_\la>0$, such that $\Omega_\la=\{x\in\X:|x|>\rho_\la\}$. This yields \eqref{eq:aconvomegaoveromega}.
\end{proof}

\begin{lem}\label{lem:newcool}
Let $b\in L^1(\R)$ be even, positive, decreasing to $0$ on the whole $\R_+$, and long-tailed function. Suppose that there exist $B, r_b,\rho_b>0$, such that
\begin{equation}\label{eq:quasisubexp}
  \int_{r_b}^\infty b(s-\tau)b(\tau)\,d\tau\leq B b(s), \quad s>\rho_b.
\end{equation}
Suppose also that
     \begin{equation}\label{old:newA5}
     \lim_{|x|\to\infty} \frac{ a (x)|x|^{d-1}}{b(|x|)}=0.
     \end{equation}
Then the inequality \eqref{eq:aconvomegaoveromega} holds for $\upomega(x):=b(|x|)$, $x\in\X$.
\end{lem}
\begin{proof}
The assumption \eqref{old:newA5} implies that
  \begin{equation}\label{eq:qr}
        g(r):= \sup_{|x|\geq r} \frac{ a (x)|x|^{d-1}}{\upomega(x)} \to 0,\quad r\to\infty.
  \end{equation}
Take an arbitrary $\delta\in(0,1)$. By \eqref{eq:qr}, one can take then  $r=r(\delta)>r_b$ such that $g(r)<\delta$. Next, by Lemma~\ref{le:rightsidelongtailedmeansspatially}, the inequality \eqref{eq:long-tailed} holds for $c=\upomega$. Therefore, there exists $\rho=\rho(\delta,r)=\rho(\delta)>\max\{r,\rho_b\}$, such that
  \begin{equation}\label{eq:1delta}
      \sup_{|y|\leq r}\frac{\upomega(x-y)}{\upomega(x)}<1+\delta, \quad |x|\geq\rho.
  \end{equation}
  Then, by \eqref{eq:qr} and \eqref{eq:1delta}, we have
  \begin{align}
    ( a *\upomega)(x) &= \upomega(x)\int_{\{|y|\leq r\}}  a (y)\frac{\upomega(x-y)}{\upomega(x)}dy+ \upomega(x)\int_{\{|y|\geq r\}} \frac{ a (y)|y|^{d-1}}{\upomega(y)}\frac{\upomega(x-y)\upomega(y)}{\upomega(x)|y|^{d-1}}dy\nonumber \\ &\leq \upomega(x) (1+\delta)\int_{|y|\leq r}  a (y)dy  + g(r)\upomega(x) \int_{\{|y|\geq r\}} \frac{b(|x-y|) b(|y|)}{b(|x|)|y|^{d-1}}dy,\nonumber\\
      \intertext{and using that $b$ is decreasing on $\R_+$ and the inequality $|x-y|\geq \bigl\lvert |x|-|y|\bigr\rvert$, one gets, cf.~\eqref{eq:normed} and recall that $g(r)<\delta$,}
      &\leq \upomega(x) (1+\delta)+\delta\upomega(x) \int_{\{|y|\geq r\}} \frac{b\bigl(\bigl\lvert |x|-|y|\bigr\rvert\bigr) b(|y|)}{b(|x|)|y|^{d-1}}dy,\nonumber\\ 
      &\leq \upomega(x) (1+\delta)+\delta\upomega(x) \sigma_d\int_r^\infty
      \frac{b\bigl(|x|-p\bigr) b(p)}{b(|x|)}dp,\label{eq:recovered}
      \end{align}
      where $\sigma_d$ is the hyper-surface area of a unit sphere in $\X$ (note that we have omitted an absolute value, as $b$ is even). Finally, using that $r>r_b$ and $\rho>\rho_b$, we obtain from \eqref{eq:quasisubexp} and \eqref{eq:recovered} that, for any $\delta\in(0,1)$,
      \[
      ( a *\upomega)(x)
      \leq \upomega(x) \bigl(1+\delta(1+\sigma_d B)\bigr), \quad |x|>\rho(\delta),
      \]
  that implies the statement.
\end{proof}

\begin{lem}\label{le:cruciallemma123}
  Let $b\in \S[1]$ be an even function. Suppose that there exists $\alpha'\in(0,1)$ such that $b^{\alpha'}\in\D$, i.e.
 \begin{equation}\label{eq:finite_d-1_momentalpha}
\int_0^\infty b(s)^{\alpha'} s^{d-1}\,ds<\infty,
\end{equation}
and, for any $\alpha\in(\alpha',1)$,
  \begin{equation}\label{eq:quickerthaninvpol}
      \lim_{|x|\to\infty} \frac{ a (x)}{b(|x|)^{\alpha}}|x|^{d-1}=0.
   \end{equation}
   Then there exists $\alpha_0\in(\alpha',1)$ such that the inequality \eqref{eq:aconvomegaoveromega} holds for  $\upomega(x)=b(|x|)^\alpha$, $x\in\X$ for all $\alpha\in(\alpha_0,1)$.
\end{lem}
  \begin{proof}
  We apply the second part of Theorem~\ref{thm:coolthm}, for $n=1$; note that then \eqref{eq:finite_d-1_momentalpha} implies \eqref{eq:finiteintd}. As a result, for any $\alpha\in(\alpha_0,1)$, the inequality \eqref{eq:ineqinsteadsubexp} holds; in particular, then \eqref{eq:quasisubexp} holds with $b$ replaced by $b^\alpha$. The latter together with \eqref{eq:quickerthaninvpol} allows to apply Lemma~\ref{lem:newcool} for $b$ replaced by $b^\alpha$, that fulfils the statement.
\end{proof}

 \begin{rem}
 Note that, by Remark~\ref{rem:saddssdadsasadsda1}, \eqref{eq:finite_d-1_momentalpha} implies that $b\in\D$ and hence, cf.~Definition~\ref{def:super-subexp}, $b\in\S[d]$.
 \end{rem}

As a result, one gets a counterpart of Proposition~\ref{prop:poldecbdd}, for the case when the function $b^+$ in \eqref{ass:newA5} decays faster than polynomial and $d>1$.

\begin{prop}\label{prop:quickerdec}
  Let \eqref{ass:newA5} hold for a function $b^+\in\D$ which is loq-equivalent to a function $b\in\S[1]$. For $d>1$, we suppose, additionally, that
\begin{equation}\label{eq:quicklydecreasing}
\lim_{s\to\infty} b(s)s^\nu =0 \qquad \text{for all $\nu\geq1$.}
\end{equation}
Then there exists $\alpha_0\in(0,1)$, such that, for all $\alpha\in(\alpha_0,1)$,
 the function $\upomega(x)=b(|x|)^\alpha$, $x\in\X$ satisfies \eqref{eq:aconvomegaoveromega}.
\end{prop}
\begin{proof}
We will use Lemma~\ref{le:cruciallemma123}.
  For $d>1$,  one gets from \eqref{eq:quicklydecreasing} that, for any $\nu >0$, there exists $\rho_\nu \geq1$, such that $b(s)\leq s^{-\nu}$, $s>\rho_\nu $. In particular, for any $\alpha'\in(0,1)$, one has \eqref{eq:finite_d-1_momentalpha}. For $d=1$, $\sigma=0$, we use instead that $b\in\S[1]$ implies \eqref{eq:Eimpliespolbdd}, and hence we get \eqref{eq:finite_d-1_momentalpha}, if only $\alpha'\in\bigl(\frac{1}{1+\delta},1\bigr)$. 

Next, for any $d\in\N$, choose an arbitrary $\alpha\in(\alpha',1)$. 
 Then, by \eqref{ass:newA5} and \eqref{eq:asfdsafsadfdsewtewtegw} applied for $b_1=b$ and $b_2=b^+$, we have that, for any $\epsilon\in(0,1-\alpha)$, there exists $\rho_\epsilon>0$, such that, for all $|x|>\rho_\eps$,
\begin{equation}\label{eq:safwwar}
  \frac{ a (x)}{b(|x|)^{\alpha}}|x|^{d-1}
  \leq b(|x|)^{1-\epsilon-\alpha}|x|^{d-1}
=\Bigl(b(|x|) |x|^{\nu}\Bigr)^{1-\epsilon-\alpha},
\end{equation}
where $\nu=\frac{d-1}{1-\epsilon-\alpha}\geq 0$, as $\alpha<1-\epsilon$. Clearly, \eqref{eq:safwwar} together with \eqref{eq:quicklydecreasing}, in the case $d>1$, imply \eqref{eq:quickerthaninvpol}, that fulfills the statement.
\end{proof}

\begin{rem}
Note that, in Proposition~\ref{prop:poldecbdd}, for the function $b$ given by \eqref{eq:polynomialb}, one can choose $\alpha'\in(0,1)$ such that \eqref{eq:finite_d-1_momentalpha} holds. The same property we have checked above for the function $b$ which satisfies assumptions of Proposition~\ref{prop:quickerdec}. As a result, by~Remark~\ref{rem:saddssdadsasadsda1}, the functions $\upomega(x)=b(|x|)^\alpha$, $x\in\X$ in these Propositions are integrable for all $\alpha\in(\alpha_0,1)$.
\end{rem}

\begin{defn}\label{def:super-subexp-tilde}
Let the set $\St[d]\subset\S[d]$, $d\in\N$ be defined as follows. Let 
$\St[1]$ be just the set $\S[1]$, whereas, for $d>1$, let $\St[d]$ be the set of all functions  $b\in\S[d]$, such that $b$ is either given by \eqref{eq:polynomialb} for some $M,\mu>0$ or $b$ satisfies \eqref{eq:quicklydecreasing}.
\end{defn}

\begin{rem}
All functions in Classes 2--4 in Subsection~\ref{subsec:examples} evidently satisfy \eqref{eq:quicklydecreasing} and hence belong to $\St[d]$. 
\end{rem}

\begin{thm}\label{thm:mainRd}
Let \eqref{ass:newA5} hold with $b^+\in\D$ which is log-equivalent to a function $b\in\St[d]$. Then there exists $\alpha_0\in(0,1)$, such that, for any $\delta\in(0,1)$ and $\alpha\in(\alpha_0,1)$, there exist $c_1=c_1(\delta,\alpha)>0$ and $\la=\la(\delta,\alpha)\in(0,1)$, such that
\begin{equation*}
a^{*n}(x)\leq c_1 (1+\delta)^{n} \min\bigl\{\la,b(|x|)^\alpha\bigr\}, \quad x\in\X.
\end{equation*}
In particular, for some $c_2=c_2(\delta,\alpha)>0$ and $s_\alpha=s_\alpha(\delta)>0$,
\begin{equation}\label{eq:KbRd}
a^{*n}(x)\leq c_2 (1+\delta)^n b(|x|)^\alpha, \quad |x|>s_\alpha, n\in\N.
\end{equation}
\end{thm}
\begin{proof}
Combining Proposition~\ref{prop:poldecbdd} and~\ref{prop:quickerdec} with Proposition~\ref{prop:coolsasd}, we get by 
\eqref{eq:new2} and \eqref{eq:uoiomegalambda} that there exist $\tilde{c_\delta}=\tilde{c_\delta}(\upomega)$, $\la=\la(\delta,\alpha)\in(0,1)$, where $\upomega(x)=b(|x|)^\alpha$, $x\in\X$, such that
\begin{equation}\label{eq:justadded}
a^{*n}(x)\leq \tilde{c}_\delta (1+\delta)^{n-1} \min\bigl\{\la,b(|x|)^\alpha\bigr\}, \quad x\in\X,
\end{equation}
that evidently yields \eqref{eq:justadded}. Since $b$ is tail-decreasing, we have that, for some $s_\alpha>0$, $b(|x|)^\alpha<\la$ for $|x|>s_\alpha$. This implies \eqref{eq:KbRd}.
\end{proof}

\begin{cor}\label{cor:mainRd}
Let $a(x)=b(|x|)$, $x\in\R$ for some $b\in \St[d]$. Then there exists $\alpha_0\in(0,1)$, such that, for any $\delta\in(0,1)$ and $\alpha\in(\alpha_0,1)$, there exist $c_{\delta,\alpha}>0$ and $s_\alpha=s_\alpha(\delta)>0$, such that, for all ,
\begin{equation*}
a^{*n}(x)\leq c_{\delta,\alpha} (1+\delta)^n a(x)^\alpha, \quad |x|>s_\alpha, n\in\N.
\end{equation*}
\end{cor}
\begin{proof}
Since, for some $\rho>0$, $b$ is decreasing on $(\rho,\infty)$ and \eqref{eq:finiteintd} holds, there exists $b^+\in\D$, such that $b^+(s)=b(s)$, $s>\rho$ and $b^+(s)\geq b(s)$, $s\in[0,\rho]$. Then one can apply Theorem~\ref{thm:mainRd}.   
\end{proof}  

\section*{Appendix: Non-local heat equation}

\setcounter{equation}{0}
\renewcommand{\theequation}{A.\arabic{equation}}

We can apply the obtained results to the study 
of the regular part of the fundamental solution to the non-local heat equation
\begin{equation}\label{eq:nlheat}
\frac{\partial}{\partial t}u(x,t)= \ka\int_\X a(x-y)\bigl(u(y,t)-u(x,t)\bigr)\,dy, \quad x\in\X,
\end{equation}
where $\ka>0$ and $0\leq a\in L^1(\X)\cap L^\infty(\X)$ is normalized, i.e.~$\int_\X a(x)\,dx=1$; see e.g. \cite{AMRT2010,BCF2011,KMPZ2016}.  
Consider an initial condition $u(x,0)=u_0(x)$, $x\in\X$ to \eqref{eq:nlheat} with $u_0$ from a space $E$ of bounded on $\X$ functions. Since the operator $Au=\ka a*u - \ka u$ in the right hand side of \eqref{eq:nlheat} is bounded on $E$, the unique solution to \eqref{eq:nlheat} is given by 
\begin{equation}\label{eq:fullsol}
u(x,t)=e^{-\ka t}\bigl((\delta_0+\phi_\ka(t))*u_0\bigr)(x),
\end{equation} 
where $\delta_0$ is the Dirac delta at $0\in\X$ and 
\begin{equation}\label{eq:regpartfundsol}
\phi_\ka(x,t):=\sum_{n=1}^\infty \frac{\ka^n t^n}{n!}a^{*n}(x), \quad x\in\X, t\geq0.
\end{equation}
Note that it was shown in \cite[Lemma 2.2]{CCR2006}, that if $a$ is a rapidly decreasing smooth function, then $\phi_k$ is indeed the solution to \eqref{eq:nlheat} with $u_0=\delta_0$. 

Now, for $d=1$, suppose that $a(x)=b(x)$, $x\in\R^1$, and $b$ satisfies the conditions of Theorem~\ref{thm:KestenonR}. Then, by \eqref{eq:Kb1}, the series in \eqref{eq:regpartfundsol} converges uniformly on finite time intervals for each $x> s_0$, and therefore, by \eqref{eq:subexpproperty},
\begin{equation*}
  \phi_\ka(x,t)\sim kt e^{\ka t} a(x), \quad x\to\infty, \ t>0.
  \end{equation*}

For $d>1$, let $a$ and $b$ satisfy the conditions of Theorem~\ref{thm:mainRd}. Then, for each $\delta>0$ and for each $\alpha<1$ close enough to $1$, 
\begin{equation}\label{eq:estfundsol}
  \phi_\ka(x,t)\leq c_{\delta,\alpha} \bigl(e^{\ka t(1+\delta)}-1\bigr)b(|x|)^\alpha, \quad |x|>s_{\alpha}, \ t>0 
\end{equation}
for some $c_{\delta,\alpha}>0$ and $s_{\alpha}=s_{\alpha}(\delta)>0$. In particular, if $a$ is radially symmetric and the conditions of Corollary~\ref{cor:mainRd} hold, then one can replace $b(|x|)$ on $a(x)$ in \eqref{eq:estfundsol}. 

Moreover, combining \eqref{eq:justadded} with \eqref{eq:fullsol}, one can get an estimate for the solution $u$ to \eqref{eq:nlheat} as well.
The further analysis of solutions to \eqref{eq:nlheat} can be found in \cite{FKT2016, FT2017c}.

\acks
Authors gratefully acknowledge the financial support by the DFG through CRC 701 ``Stochastic
Dynamics: Mathematical Theory and Applications'' (DF~and~PT), the European
Commission under the project STREVCOMS PIRSES-2013-612669 (DF), and the ``Bielefeld Young Researchers'' Fund through the Funding Line Postdocs:
``Career Bridge Doctorate\,--\,Postdoc'' (PT).

\end{document}